\documentclass[11pt,reqno]{elsarticle}

\makeatletter
\def\ps@pprintTitle{%
 \let\@oddhead\@empty
 \let\@evenhead\@empty
 \def\@oddfoot{}%
 \let\@evenfoot\@oddfoot}
\makeatother

\usepackage{amsmath,amsthm,amscd,wrapfig,amsfonts,amssymb}
\usepackage{enumitem,comment,mathabx}
\usepackage{booktabs,multirow,lscape,rotating}

\usepackage{url,parskip}
\usepackage[top=22mm, bottom=22mm, left=22mm, right=22mm]{geometry}

\newtheorem{theorem}{Theorem}

\newtheorem{remark}{Remark}

\newdimen\CdotAxis
\newcommand*{\CdotAux}[3]{%
  {%
    \settoheight\CdotAxis{$#2\vcenter{}$}%
    \sbox0{%
      \raisebox\CdotAxis{%
        \scalebox{#1}{%
          \raisebox{-\CdotAxis}{%
            $\mathsurround=0pt #2#3$%
          }%
        }%
      }%
    }%
    \dp0=0pt %
    \sbox2{$#2\bullet$}%
    \ifdim\ht2<\ht0 %
      \ht0=\ht2 %
    \fi
    \sbox2{$\mathsurround=0pt #2#3$}%
    \hbox to \wd2{\hss\usebox{0}\hss}%
  }%
}

\DeclareMathOperator{\sign}{sign}
\DeclareMathOperator{\erf}{erf}

\providecommand{\Lp}[1]{\ensuremath{L^{#1}\left(\mathbb{R}\right)}}

\providecommand{\Sr}{S(\mathbb{R})}
\providecommand{\Spr}{S'(\mathbb{R})}

\providecommand{\e}[1]{\ensuremath{\times 10^{#1}}}

\makeatletter
\g@addto@macro\normalsize{%
  \setlength\abovedisplayskip{.4em}
  \setlength\belowdisplayskip{.4em}
  \setlength\abovedisplayshortskip{.4em}
  \setlength\belowdisplayshortskip{.4em}
}

\begin{document}

\begin{frontmatter}
\begin{abstract}
We describe an algorithm for the numerical solution of 
second order linear ordinary differential equations in the high-frequency 
regime.  It is founded on the recent observation
that solutions of equations of this type can be 
accurately represented using nonoscillatory phase functions.
Unlike standard solvers for ordinary differential equations,
the running time of our algorithm is independent of the frequency 
of oscillation of the solutions.  
We illustrate this and other properties  of the method with numerical experiments.

\end{abstract}

\begin{keyword}
Ordinary differential equations \sep
fast algorithms \sep
phase functions \sep
special functions \sep
Bessel's equation
\end{keyword}

\title
{
On the numerical solution of 
second order ordinary differential equations in the high-frequency regime
}

\author[jb]{James Bremer\corref{cor1}}

\cortext[cor1]{Corresponding author (bremer@math.ucdavis.edu)}

\address[jb]{Department of Mathematics, University of California, Davis}

\end{frontmatter}

\begin{section}{Introduction}

Second order linear differential equations of the form
\begin{equation}
y''(t) + \lambda^2 q(t) y(t) = 0
\ \ \ \mbox{for all}\ \ a \leq t \leq b
\label{introduction:second_order}
\end{equation}
are ubiquitous in analysis and mathematical physics.  As a consequence, much
attention has been devoted to the development of numerical algorithms for their solution
and, in most regimes, fast and accurate methods are available.

However, when $q$ is positive and $\lambda$ is real-valued and large,
the solutions of  
(\ref{introduction:second_order}) are highly oscillatory
(this is a consequence of the Sturm comparison theorem)
and standard solvers for ordinary 
differential equations (for instance, Runge-Kutta schemes and spectral methods) suffer.
Specifically, their  running times 
grow linearly with the parameter $\lambda$,
which makes them prohibitively expensive when $\lambda$ is large.

Because of the poor performance of standard solvers,
asymptotic methods are often used in this regime.
In some instances, they allow for 
the accurate evaluation of 
solutions of equation of the form (\ref{introduction:second_order})
 using a number of operations which is independent 
of the parameter $\lambda$.  For example, \cite{Bogaert-Michiels-Fostier} presents
an $\mathcal{O}(1)$ algorithm for calculating  Legendre polynomials of arbitrary
order  using a combination of direct
evaluation and asymptotic formulas; it achieves near machine precision accuracy
and serves as the basis for an efficient parallel algorithm (also presented
in \cite{Bogaert-Michiels-Fostier}) for the 
construction of Gauss-Legendre quadratures of extremely large orders.

The formulas used in \cite{Bogaert-Michiels-Fostier} are particular
to Legendre polynomials, and while the same approach can be 
applied to other classes of special functions defined by equations
of the form (\ref{introduction:second_order}), in each case
a new, specialized approach  must be devised. 
Indeed, despite the extensive existing literature on the asymptotic
approximation of Legendre polynomials, the algorithm of \cite{Bogaert-Michiels-Fostier}
 required the development  of a novel asymptotic expansion with suitable numerical
properties.

Here, we describe an algorithm for the numerical solution of second order
linear ordinary differential equations of the form (\ref{introduction:second_order})
whose running time is independent of the parameter $\lambda$.  
It  applies to a large class of second order
ordinary differential equations ---
which includes those defining Bessel functions, Legendre functions of noninteger
orders, prolate spheriodal wave functions and the classical orthogonal polynomials ---
and its only inputs are the parameter $\lambda$ and the values of the function
$q$ at a collection of points on the interval $[a,b]$ on which
(\ref{introduction:second_order}) is given. 

Our approach proceeds by constructing a nonscillatory phase
function which represents solutions of (\ref{introduction:second_order}).
We say that $\alpha$ is a phase function
for (\ref{introduction:second_order}) if the functions $u,v$ defined by the formulas
\begin{equation}
u(t) = \frac{\cos(\alpha(t)) }{\left|\alpha'(t)\right|^{1/2}},
\label{introduction:u}
\end{equation}
and
\begin{equation}
v(t) = \frac{\sin(\alpha(t)) }{\left|\alpha'(t)\right|^{1/2}}
\label{introduction:v}
\end{equation}
comprise
a basis in the space of solutions of (\ref{introduction:second_order}).
Phase functions play a key role in the theories of special functions and
global transformations of ordinary differential equations
\cite{Boruvka,Neuman,NISTHandbook,Andrews-Askey-Roy}, and
are the basis of many numerical algorithms
(see \cite{Spigler-Vianello, Goldstein-Thaler,Heitman-Bremer-Rokhlin-Vioreanu}
for  representative examples).

It was observed by E.E. Kummer in \cite{Kummer}  that $\alpha$ is  a phase function for 
(\ref{introduction:second_order})  if and only if it satisfies the 
third order nonlinear differential equation
\begin{equation}
\left(\alpha'(t)\right)^2 = \lambda^2q(t) - \frac{1}{2}\frac{\alpha'''(t)}{\alpha'(t)}
+ \frac{3}{4} \left(\frac{\alpha''(t)}{\alpha'(t)}\right)^2
\label{introduction:kummers_equation}
\end{equation}
on the interval $[a,b]$.  
The presence of quotients in (\ref{introduction:kummers_equation})
is often inconvenient, and we prefer the more tractable equation
\begin{equation}
r''(t) - \frac{1}{4} \left(r'(t)\right)^2 + 4\lambda ^2 \left( \exp(r(t)) - q(t)\right) = 0 
\label{introduction:kummers_equation_logarithm}
\end{equation}
obtained from (\ref{introduction:kummers_equation}) by letting
\begin{equation}
\alpha'(t) = \lambda \exp\left(\frac{r(t)}{2}\right).
\end{equation}
Of course, if $r$ is a solution of (\ref{introduction:kummers_equation_logarithm})
then a solution $\alpha$ of (\ref{introduction:kummers_equation})
is given by the formula
\begin{equation}
\alpha(t) = \lambda \int_a^t \exp\left(\frac{r(u)}{2}\right)\ du.
\label{introduction:alphafromr}
\end{equation}
We will refer to (\ref{introduction:kummers_equation}) as Kummer's equation
and (\ref{introduction:kummers_equation_logarithm}) as  the logarithm form
of Kummer's equation.
The form of these equations
 and the appearance
of $\lambda$ in them suggests that their solutions will be oscillatory --- and most of them are.
However, there are several well-known examples of 
second order ordinary differential equations
which admit nonoscillatory phase functions.  For example, the function
\begin{equation}
\alpha(t) = \lambda \arccos(t)
\end{equation}
is a phase function for  Chebyshev's equation
\begin{equation}
y''(t) + \left(\frac{2 + t^2 + 4 \lambda^2 (1-t^2)}{4(1-t^2)^2} \right) y(t) = 0
\ \ \ \mbox{for all} \ \ -1 \leq t \leq 1.
\label{introduction:chebyshev}
\end{equation}
Its existence is the basis of the fast Chebyshev
transform and several other widely used numerical algorithms
(see, for instance, \cite{Trefethen}).
 Bessel's equation
\begin{equation}
y''(t) + \left(1-\frac{\lambda^2-1/4}{t^2} \right) y(t) = 0
\ \ \ \mbox{for all}\ \ 0 < t < \infty
\label{introduction:bessel}
\end{equation}
also admits a nonoscillatory phase function, although it cannot be 
expressed via elementary functions 
\cite{Heitman-Bremer-Rokhlin-Vioreanu}.

Exact solutions of  (\ref{introduction:kummers_equation}) 
 which are nonoscillatory need not  exist in the general case.  
However, \cite{Heitman-Bremer-Rokhlin} and \cite{Bremer-Rokhlin} make the observation that
when the coefficient $q$ appearing in (\ref{introduction:second_order}) is nonoscillatory,
there exists a nonoscillatory function $\alpha$ such that 
(\ref{introduction:u}), (\ref{introduction:v}) approximate
solutions (\ref{introduction:second_order})
 with accuracy on the order of $(\mu\lambda)^{-1}\exp(-\mu \lambda)$,
where $\mu$ is a constant which depends on the coefficient $q$
 but not on $\lambda$.       More specifically, there is a nonoscillatory
function $r$ which is a solution of the equation
\begin{equation}
r''(t) - \frac{1}{4} \left(r'(t)\right)^2 + 4\lambda ^2 \left( \exp(r(t)) - q(t)\right) = 
q(t) \nu,
\label{introduction:r2eq}
\end{equation}
where $\nu$ is a smooth function such that
\begin{equation}
\left\|\nu\right\|_\infty = \mathcal{O}\left(\frac{1}{\mu}\exp(-\mu\lambda)\right).
\label{introduction:nu}
\end{equation}
The function $\alpha$ obtained from $r$ via formula  (\ref{introduction:alphafromr})
is a solution of the nonlinear differential equation
\begin{equation}
\left(\alpha'(t)\right)^2 = \lambda^2 
\left(\frac{\nu(t)}{4\lambda^2}+1\right)q(t) - \frac{1}{2}\frac{\alpha'''(t)}{\alpha'(t)}
+ \frac{3}{4} \left(\frac{\alpha''(t)}{\alpha'(t)}\right)^2
\ \ \ \mbox{for all}\ \ t \in\mathbb{R};
\label{introduction:alpha2}
\end{equation}
this implies that $\alpha$ is a phase function
for the equation
\begin{equation}
y''(t) + \lambda^2 \left(1+\frac{\nu(t)}{4\lambda^2}\right)q(t)y(t) = 0
\ \ \ \mbox{for all}\ \ a \leq t \leq b.
\label{introduction:second_order2}
\end{equation}
%
It follows from (\ref{introduction:second_order2}) and (\ref{introduction:nu})
that when $\alpha$ is inserted into formulas 
(\ref{introduction:u}) and (\ref{introduction:v}), 
the resulting functions approximate solutions of 
(\ref{introduction:second_order}) with 
$\mathcal{O}\left((\mu\lambda)^{-1}\exp(-\mu \lambda)\right)$
accuracy (see Theorem~12 in \cite{Bremer-Rokhlin}).
The functions $r$ and $\alpha$ are nonoscillatory
in the sense that then can be accurately represented
using various series expansions (e.g., expansions in Chebyshev polynomials)
whose number of terms does not depend on $\lambda$.
In other words,  $\mathcal{O}(1)$ terms are required to represent the 
 solutions of (\ref{introduction:second_order}) with 
 $\mathcal{O}\left((\mu\lambda)^{-1}\exp(-\mu\lambda)\right)$ accuracy.
This is an improvement over superasymptotic and hyperasymptotic expansions
(see, for instance,  \cite{Daalhuis-Olver,Daalhuis}),
which represent
solutions of (\ref{introduction:second_order}) with accuracy on the order
of $\exp(-\rho \lambda)$ using expansions with $\mathcal{O}(\lambda)$ terms.
Theorem~12 of \cite{Bremer-Rokhlin},
which is reproduced as  Theorem~\ref{main_theorem} in Section~\ref{section:nonoscillatory}
of this article, 
gives a precise statement regarding the existence of nonoscillatory phase
functions.

The method used to establish the existence of nonoscillatory
phase functions in  \cite{Bremer-Rokhlin} is constructive and could serve as the basis
of a numerical method for their computation.  There are, however, at least
two serious drawbacks to such an approach:  an algorithm based on
the method of \cite{Bremer-Rokhlin} would require 
that the coefficient $q$ be extended to the real line as well as knowledge
of the  first two derivatives of $q$.  

In this article, we describe a method
for constructing a solution of the logarithm form of Kummer's equation
whose difference from the nonoscillatory solution of 
(\ref{introduction:r2eq})
is on the order of $\exp\left(-\frac{1}{2}\mu\lambda\right)$.
It does not require that $q$ be extended beyond the interval
$[a,b]$, nor does it take as input the values of the derivatives of $q$.
Indeed,  its only inputs are the parameter $\lambda$
and the values of $q$ at a collection of points on the interval $[a,b]$.

Our approach is based on two observations.  First, that if
\begin{equation}
q(t) = 1 + \epsilon(t)
\label{introduction:qalmost1}
\end{equation}
for all $t$ in an interval of the form $[a,a+\tau]$, where $\tau > 0$ and
$\epsilon$
is a smooth function of sufficiently small magnitude,
then the difference between the solution of the boundary value problem
\begin{equation}
\left\{
\begin{aligned}
r''(t) - \frac{1}{4} \left(r'(t)\right)^2 + 4\lambda^2\left(\exp(r(t)) - q(t)\right) &= 0 
\ \ \mbox{for all} \ \ a \leq t \leq b
\\
r(a) = r'(a) &= 0
\end{aligned}
\right.
\label{introduction:bvp1}
\end{equation}
and the nonoscillatory solution 
of (\ref{introduction:r2eq}) 
is on the order of  $\exp\left(-\frac{1}{2} \mu \lambda\right)$ on the interval $[a,b]$.
Second,  that when the coefficient $q$ is perturbed by a 
smooth function $\varphi$ which is of sufficiently small magnitude on the interval
$[a,b]$,
the changes induced in the restrictions of the nonoscillatory 
solution of (\ref{introduction:r2eq}) 
and its derivative to the interval $[a,b]$ 
are on the order of
$\exp\left(- \mu \lambda\right)$.
%
Both of these observations are obtained by 
combining  Theorem~\ref{bound:theorem5} of Section
\ref{section:bound} and Theorem~\ref{continuity:theorem3}
of Section~\ref{section:continuity},
 which are the two principal results of this article.

We exploit these observations  as follows.
First, we  construct a windowed version $\tilde{q}$ of the function $q$ such that
\begin{equation}
\tilde{q}(t) = \begin{cases}
1 + \epsilon(t)   & \mbox{for all}\ \ a \leq t \leq b-\Delta\\
q(t)              & \mbox{for all}\ \ b-\Delta < t \leq b,
\end{cases}
\label{introduction:qtilde1}
\end{equation}
where $\Delta$ is a small positive real number and
$\epsilon(t)$ is a function of small magnitude,
and calculate a solution $r_1$ of the initial value problem
\begin{equation}
\left\{
\begin{aligned}
r_1''(t) - \frac{1}{4} \left(r_1'(t)\right)^2 
+ 4\lambda^2\left(\exp\left(r_1(t)\right) - \tilde{q}(t)\right) &= 0
\ \ \ \mbox{for all}\ \ a \leq t \leq b
 \\
r_1(a)=r_1'(a) &= 0.
\end{aligned}
\right.
\label{introduction:bvp2}
\end{equation}
Next, we obtain a solution $r_2$ of the problem
\begin{equation}
\left\{
\begin{aligned}
r_2''(t) - \frac{1}{4} \left(r_2'(t)\right)^2 
+ 4\lambda^2\left(\exp\left(r_2(t)\right) - q(t)\right) &= 0
\ \ \ \mbox{for all}\ \ a \leq t \leq b
 \\
r_2\left(b\right)= r_1\left(b\right)
\ \ \ \mbox{and}\ \  \
r_2'\left(b\right)&= r_1\left(b\right).
\end{aligned}
\right.
\label{introduction:bvp3}
\end{equation}
From our first observation, we see that the difference
between the solution of (\ref{introduction:bvp2}) and the nonoscillatory 
function $\tilde{r}$ obtained by applying Theorem~\ref{main_theorem}
to the equation
\begin{equation}
\tilde{r}''(t) - 
\frac{1}{4} \left(\tilde{r}'(t)\right)^2 + 4\lambda ^2 \left( \exp(\tilde{r}(t)) - \tilde{q}(t)
\right) = 0,
\label{introduction:r1}
\end{equation}
is on the order of  $\exp\left(-\frac{1}{2} \mu \lambda\right)$.
Moreover, according to our second observation, the difference between the
the function $r_1$ 
and the nonoscillatory solution $r$ of 
\begin{equation}
r''(t) - \frac{1}{4} \left(r'(t)\right)^2 + 4\lambda ^2 \left( \exp(r(t)) - q(t)\right) = 
q(t) \nu(t)
\label{introduction:r2}
\end{equation}
whose existence is guaranteed by Theorem~\ref{main_theorem} 
is on the order of 
\begin{equation}
\exp\left(-\frac{1}{2} \mu \lambda\right)
\label{introduction:r2error}
\end{equation}
on the interval $\left[b-\Delta,b\right]$,
as is the difference between the derivatives of these two functions.
In particular, 
\begin{equation}
\left|r_1\left(b\right) -r\left(b\right) \right|
+
\left|r_1'\left(b\right) - r'\left(b\right)\right|
= O\left(\exp\left(-\frac{1}{2}\mu\lambda\right)\right).
\label{r1order}
\end{equation}
Together (\ref{introduction:nu}), (\ref{introduction:bvp3}),
(\ref{introduction:r2}), and (\ref{r1order})
imply that
\begin{equation}
\left| r_2(t) - r(t) \right| = O\left( \exp\left(-\frac{1}{2}\mu\lambda\right)\right)
\end{equation}
for all $t \in [a,b]$.
That is, the difference between the solution $r_2$ of the boundary value problem 
(\ref{introduction:bvp3})
and  the nonoscillatory solution $r$ of
(\ref{introduction:r2eq}) decays exponentially with $\lambda$.  

In the high-frequency regime, the difference between $r_1$ and the nonoscillatory
function $\tilde{r}$
is considerably smaller than machine precision,
as is the difference between $r_2$ and the nonoscillatory function $r$.
Consequently, for the purposes of numerical computation, $r_1$ and
$r_2$  can be regarded as nonoscillatory.  In particular, 
solutions of the boundary value problems
(\ref{introduction:bvp2}) and (\ref{introduction:bvp3})
can be obtained via a standard method
for the numerical solution of ordinary differential equations,
and each of the functions $r_1$ and $r_2$ can be approximated to high accuracy by
a finite series expansion whose number of terms 
does not depend on $\lambda$.  Moreover, the number of operations
required to compute these expansions of $r_1$ and $r_2$ does not depend on $\lambda$.

There is one significant limitation on the accuracy
obtained by the algorithm of this paper.  When $\lambda$ is large, the
evaluation of the functions $u$, $v$ defined via the formulas
(\ref{introduction:u}) and (\ref{introduction:v})
requires the computation of trigonometric functions of large
arguments.  There is an inevitable loss of accuracy when these calculations are performed
in finite precision arithmetic.  Nonetheless, acceptable accuracy is obtained
in many cases. For instance, Section~\ref{section:experiments:bessel}
describes an  experiment in which the Bessel function of the first kind of order $10^8$
was evaluated at a large collection of points on the real axis to approximately
ten digits of accuracy.

We also note that while some accuracy is lost when evaluating solutions
of (\ref{introduction:second_order}) when $\lambda$ is large,  
the phase functions produced by the algorithm
of this paper are highly accurate is all cases.
 Among other things, they can be used
to rapidly calculate the roots of special functions to extremely
high precision.
This and other applications of nonoscillatory phase functions will be reported at a later date.

The remainder of this paper is organized as follows.    Section~\ref{section:preliminaries}
summarizes a number of mathematical and numerical facts to be used in the rest of the paper.
  In Section~\ref{section:apparatus}, we develop the analytic
apparatus used in 
Section~\ref{section:algorithm} to develop an 
 algorithm for the rapid solutions of second order linear differential
equations in the high-frequency regime.   Section~\ref{section:experiments}
presents the results of numerical experiments conducted to assess the performance
of the algorithm of Section~\ref{section:algorithm}.

\end{section}

\begin{section}{Analytic and numerical preliminaries}

%
%

\begin{subsection}{Schwartz functions and tempered distributions}


We say that an infinitely differentiable function
$\varphi: \mathbb{R} \to \mathbb{C}$ is a Schwartz function if $\varphi$ and 
all of its derivatives decay faster 
than any polynomial.  That is, if 
\begin{equation}
\sup_{t\in\mathbb{R}} |t^i \varphi^{(j)}(t)| < \infty
\end{equation}
for all pairs $i,j$ of nonnegative integers.  The set  of all Schwartz functions 
is denoted by  $\Sr$.   We endow it with the topology generated by the
family of seminorms
\begin{equation}
\|\varphi\|_{k} =\sum_{j=0}^k\ \sup_{t\in\mathbb{R}} \left|t^k \varphi^{(j)}(x)\right|
\ \ \ k=0,1,2,\ldots,
\end{equation}
so that a sequence $\{\varphi_n\}$ of functions in $\Sr$ converges
to $\varphi$ in $\Sr$ if and only if
\begin{equation}
\lim_{n\to\infty}\|\varphi_n - \varphi\|_{k} = 0 \ \ \ \mbox{for all}\ \ k=0,1,2,\ldots.
\end{equation}
We denote the space of continuous linear functionals on $\Sr$, which are known
as tempered distributions, by $\Spr$.


See, for instance, \cite{HormanderI} for a thorough discussion of
Schwartz functions and tempered distributions.

\label{section:preliminaries:schwartz}
\end{subsection}

%

\begin{subsection}{The Fourier transform}
We define the Fourier transform of a function $f \in \Sr$ via the formula
\begin{equation}
\widehat{f}(\xi) = \int_{-\infty}^{\infty} \exp(-i x\xi)f(x)\ dx.
\label{preliminaries:fourier:1}
\end{equation}
The Fourier transform is an isomorphism $\Sr \to \Sr$ (meaning that it
is a continuous, invertible mapping $\Sr \to \Sr$ whose inverse is also
continuous).  
The formula
\begin{equation}
\left<\widehat{\omega},\varphi\right> = \left<\omega,\widehat{\varphi}\right>
\label{preliminaires:fourier:2}
\end{equation}
 extends the Fourier transform to an isomorphism $\Spr \to \Spr$.
The definition (\ref{preliminaires:fourier:2})
coincides with (\ref{preliminaries:fourier:1}) when $f \in \Lp{1}$.
Moreover, when $f \in \Lp{2}$,
\begin{equation}
\widehat{f}(\xi) = \lim_{R\to\infty} \int_{-R}^{R} \exp(-i x\xi)f(x)\ dx.
\label{preliminaries:fourier:3}
\end{equation}
Owing to our choice of convention for the Fourier transform,
\begin{equation}
\widehat{f*g}(\xi) = \widehat{f}(\xi) \widehat{g}(\xi)
\label{preliminaries:fourier:4}
\end{equation}
and
\begin{equation}
\widehat{f \cdot g}(\xi) = \frac{1}{2\pi} \int_{-\infty}^\infty \widehat{f}(\xi-\eta)
\widehat{g}(\eta)\ d\eta
\label{preliminaries:fourier:5}
\end{equation}
whenever $f$ and $g$ are elements of $\Lp{1}$.
Moreover, 
\begin{equation}
f(x) = \frac{1}{2\pi}\int_{-\infty}^\infty \exp(ix\xi) \widehat{f}(\xi)\ d\xi
\end{equation}
whenever $f$ and $\widehat{f}$ are elements of $\Lp{1}$.  
The observation that $f$ is an entire function when
 $\widehat{f}$ is a compactly supported distribution
 is one consequence of the well-known Paley-Wiener theorem.
See \cite{GrafakosC, Grafakos} for a thorough treatment
of the Fourier transform.

\end{subsection}

%
%

\begin{subsection}{The constant coefficient Helmholtz equation}

The following theorem is a special case of a more general one which can be 
found in \cite{HormanderII}.

\vskip 1em
\begin{theorem}
Suppose that $f \in \Sr$.
If $\lambda$ is a positive real number, 
then the function $g$ defined by the formula
\begin{equation}
g(x) = \frac{1}{2\lambda} \int_{-\infty}^\infty\sin \left(\lambda \left|x-y\right|\right) f(y)\ dy
\label{preliminaries:helmholtz:g}
\end{equation}
is an infinitely differentiable function,
\begin{equation}
g''(x) + \lambda^2 g(x) = f(x) \ \ \mbox{for all} \ x\in\mathbb{R},
\label{preliminaries:helmholtz:diffeq}
\end{equation}
and 
\begin{equation}
\widehat{g}(\xi) = 
\frac{\widehat{f}(\xi)}{\lambda^2 - \xi^2}.
\label{preliminaries:helmholtz:tempdist}
\end{equation}
If $\lambda$ is complex number with positive imaginary part,
then the function $h$ defined by the formula
\begin{equation}
h(x) = \frac{1}{2\lambda i } 
\int_{-\infty}^\infty\exp \left(2 \lambda i\left|x-y\right|\right) f(y)\ dy
\label{preliminaries:helmholtz:h}
\end{equation}
is an infinitely differentiable function,
\begin{equation}
h''(x) + \lambda^2 h(x) = f(x) \ \ \mbox{for all} \ x\in\mathbb{R},
\end{equation}
and 
\begin{equation}
\widehat{h}(\xi) = 
\frac{\widehat{f}(\xi)}{\lambda^2 - \xi^2}.
\end{equation}

\label{preliminaries:helmholtz:theorem1}
\end{theorem}
We interpret the Fourier transform (\ref{preliminaries:helmholtz:tempdist})
of $g$  as a tempered distribution
defined via principal value integrals; 
that is to say that for all $\varphi \in \Sr$,
\begin{equation}
\left<\frac{\widehat{f}(\xi)}{\lambda^2-\xi^2}, \varphi \right>
=
\frac{1}{2\lambda}
\left(
\lim_{\epsilon \to 0}
\int_{|\xi-\lambda| > \epsilon} 
\frac{\widehat{f}(\xi) \varphi(\xi)}{\lambda-\xi}\ d\xi
-
\lim_{\epsilon \to 0}
\int_{|\xi+\lambda| > \epsilon} 
\frac{\widehat{f}(\xi)\varphi(\xi)}{\lambda+\xi}\ d\xi
\right).
\end{equation}

The following variant of Theorem~\ref{preliminaries:helmholtz:theorem1}
can be found in \cite{Coddington-Levinson}.
\vskip 1em
\begin{theorem}
Suppose that $f$ is continuous on the interval $[a,b]$, and that
 $\lambda$
 is a positive real number.  Suppose also that $y:[a,b]\to\mathbb{C}$ is twice
continuously differentiable, and that
\begin{equation}
y''(x) + \lambda^2 y(x) = f(x)
\ \ \ \mbox{for all}\ \ a \leq x \leq b.
\end{equation}
Then
\begin{equation}
y(x) = y(a) + y'(a) (x-a) + \frac{1}{\lambda}
\int_a^x 
\sin\left(\lambda\left(x-u\right)\right) f(u)\ du
\ \ \ \ \mbox{for all} \ \ a \leq x \leq b.
\end{equation}
\label{preliminaries:helmholtz:theorem2}
\end{theorem}

\label{section:preliminaries:helmholtz}
\end{subsection}

%
%
\begin{subsection}{Schwarzian derivatives}
The Schwarzian derivative of a smooth function $f: \mathbb{R} \to \mathbb{R}$ 
is
\begin{equation}
\{f,t\} = \frac{f'''(t)}{f'(t)} - \frac{3}{2} \left(\frac{f''(t)}{f'(t)}\right)^2.
\end{equation}
The Schwarzian derivative of $x(t)$
is related to the Schwarzian derivative of its inverse $t(x)$ through
the formula
\begin{equation}
\{x,t\} = - \left(\frac{dx}{dt}\right)^2\{t,x\}.
\label{preliminaries:Schwarzian_derivative:change_of_vars}
\end{equation}
This identity
can be found, for instance, in Section 1.13 of \cite{NISTHandbook}.

\label{section:preliminaries:schwarzian_derivative}
\end{subsection}

\begin{subsection}{Gronwall's inequality}
The following well-known inequality can be found in, for instance,  \cite{Bellman}.
\vskip 1em
\begin{theorem}
Suppose that $f$ and $g$ are continuous functions on the interval $[a,b]$ such that
\begin{equation}
f(t) \geq 0 \ \ \mbox{and}\ \ g(t) \geq 0 \ \ \ \mbox{for all}\ \ a \leq t \leq b.
\end{equation}
Suppose further that there exists a real number $C>0$ such that
\begin{equation}
f(t) \leq C + \int_a^t f(s)g(s)\ ds \ \ \ \mbox{for all}\ \ a \leq t \leq b.
\end{equation}
Then
\begin{equation}
f(t) \leq C \exp\left(\int_a^t g(s)\ ds\right)\ \ \ \mbox{for all}\ \ a \leq t \leq b.
\end{equation}
\label{preliminaries:gronwall:theorem1}
\end{theorem}
\label{section:preliminaries:gronwall}
\end{subsection}

%
%
\begin{subsection}{The Lambert {$W$} function}
The Lambert $W$ function or product logarithm is the multiple-valued
inverse of the function
\begin{equation}
f(z) = z \exp(z).
\end{equation}
We follow \cite{knuth} in using $W_{0}$ to denote 
the branch of $W$ which is real-valued and greater than or equal to 
$-1$ on the interval $[-1/e,\infty)$.
The following elementary fact concerning $W_0$ can be found in \cite{knuth}.

\vskip 1em
\begin{theorem}
Suppose that $y \geq -1/e$ is a real number.  Then
\begin{equation}
x\exp(x) \leq y 
\ \ \ \mbox{if and only if}\ \ \
x \leq W_0(y).
\end{equation}
\label{preliminaries:LambertW:theorem1}
\end{theorem}




\label{section:preliminaries:LambertW}
\end{subsection}

%
%

\begin{subsection}{The Picard-Lindel\"of Theorem}
The following well-known theorem  can be found in \cite{Coddington-Levinson},
among many other sources.

\vskip 1em
\begin{theorem}
Suppose that  $\Omega$ is a domain in $\mathbb{R}^n$, that 
the function $f:[t_0,t_1] \times \Omega  \to \mathbb{R}^n$ is continuous,
and that there exists a $M$ such that
\begin{equation}
\left|f(t,x) - f(t,y)\right| \leq M |x-y|
\end{equation}
for all $x,y \in \Omega$ and $t \in [t_0,t_1]$.  Suppose also that
$t_0 \leq t \leq t_1$, and that $y_0$ is an element of $\Omega$.  
Then there exist a positive real number $\epsilon$,
and a differentiable function $y:[t_0,t_0+\epsilon] \to \mathbb{C}$
such that
\begin{equation}
y'(t) = f(t,y(t))
\ \ \ \mbox{for all}\ \ t_0 \leq t \leq t_0 + \epsilon,
\end{equation}
and $y(t_0) = y_0$.
\label{preliminaries:picard:theorem1}
\end{theorem}

\label{section:preliminaries:picard}
\end{subsection}

%
%

\begin{subsection}{Chebyshev polynomials and interpolation}

The $(m+1)$-point Chebyshev grid on the interval $[-1,1]$ is the set 
\begin{equation}
\left\{
\cos\left(\frac{\pi j}{m}\right) :  \ \ j=0,1,\ldots,m
\right\}.
\label{preliminaries:chebyshev:nodes}
\end{equation}
We refer to each element of (\ref{preliminaries:chebyshev:nodes})
as a Chebyshev node or point.

Suppose that $f:[-1,1]\to\mathbb{R}$ is a Lipschitz continuous function.  For
each integer $m$, there exists a unique polynomial of degree $m$
which agrees with $f$ on the $(m+1)$-point Chebyshev grid.
We refer to this polynomial as the $m^{th}$ order Chebyshev interpolant for $f$ and denote 
it  by $\Psi_m\left[f\right]$.  In other words, $\Psi_m\left[f\right]$ is 
the polynomial of degree $m$ defined by the requirement
that
\begin{equation}
\Psi_m\left[f\right]
\left(\cos\left(\frac{\pi j}{m}\right)\right) = f\left(\cos\left(\frac{\pi j}{m}\right)\right)
\end{equation}
for $j=0,1,2,\ldots,m$.  Moreover,
$\Psi_m\left[f\right]$ converges to $f$ in $L^\infty\left(\left[-1,1\right]\right)$
norm as $m \to \infty$.
The following two theorems  provide 
estimates on the rate of convergence of $\Psi_m\left[f\right]$
in $L^\infty\left(\left[-1,1\right]\right)$
under additional assumptions on the function $f$.
The observation that $\Psi_m\left[f\right] \to f$ when 
$f$ is Lipschitz continuous and proofs of the following two theorems can be
found in  \cite{Trefethen}, among many other sources.
\vskip 1em
\begin{theorem}
Suppose that $p$ is a positive integer and 
$f \in C^p\left(\left[-1,1\right]\right)$.  Then 
\begin{equation}
\left\|
\Psi_m\left[f\right] - f 
\right\|_\infty
= O \left(m^{-p}\right).
\end{equation}
\label{preliminaries:chebyshev:theorem1}
\end{theorem}

\begin{theorem}
Suppose that $f$ is analytic on an ellipse with foci $\pm 1$
the sum of whose semiaxes is $\rho$.  Then
\begin{equation}
\left\|
\Psi_m\left[f\right] - f 
\right\|_\infty
= O \left(\rho^{-m}\right).
\end{equation}
\label{preliminaries:chebyshev:theorem2}
\end{theorem}

Given the values of $f$ on the $(m+1)$-point Chebyshev grid
$x_0,x_1,\ldots,x_m$, the value of $\Psi\left[f\right]$ can be calculated
at any point $x$ in $[-1,1]$ via the formula
\begin{equation}
\Psi\left[f\right](x) = 
\left(\sum_{j=0}^m \frac{(-1)^j f(x_j)}{x-x_j}\right)
\Big/
\left(\sum_{j=0}^m \frac{(-1)^j}{x-x_j}\right).
\label{preliminaries:chebyshev:barycentric}
\end{equation}
The process of approximating a function $f$ via $\Phi\left[f\right]$
 is referred to 
as Chebyshev interpolation and (\ref{preliminaries:chebyshev:barycentric})
is known as the barycentric interpolation formula for Chebyshev polynomials.
The stability of barycentric interpolation is discussed extensively in \cite{Trefethen}.

Suppose that $f:[-1,1]\to\mathbb{R}$ is continuous function, 
that $m$ is a positive integer, and that 
 $g$ is the function defined by the formula
\begin{equation}
g(t) = \int_{-1}^t \Psi_m\left[f\right](u)\ du.
\end{equation}
If $\alpha = \{\alpha_0,\alpha_1,\ldots,\alpha_m\}$ is the vector
defined by the formula
\begin{equation}
\alpha_j =  f\left(\cos\left(\frac{\pi j}{m}\right)\right)
\end{equation}
and $\beta = \{\beta_0,\beta_1,\ldots,\beta_m\}$ is the vector
defined by the formula
\begin{equation}
\beta_j = g\left(\cos\left(\frac{\pi j}{m}\right)\right),
\end{equation}
then we refer to the  $(m+1)\times (m+1)$ matrix $S_m$ such that
\begin{equation}
S_m \alpha = \beta
\end{equation}
as the spectral integration matrix of order $m$
(that such a matrix exists is clear since the underlying operation
is linear).  

The preceding constructions can be easily modified 
in order to accommodate functions defined
on any finite interval $[a,b]$.  For instance,  if we denote
the points in the $(m+1)$-point Chebyshev grid on $[-1,1]$
by $\rho_0,\rho_1,\ldots,\rho_m$, then the $(m+1)$-point Chebyshev grid on the interval $[a,b]$ is 
the set
\begin{equation}
\left\{
\frac{a-b}{2}\ \rho_j + \frac{a+b}{2} :\ 
j=0,1\ldots,m.
\right\}
\end{equation}
\vskip 1em
\begin{remark}
The set (\ref{preliminaries:chebyshev:nodes})
is the collection  of  the extreme points of the $m^{th}$ order
Chebyshev polynomial $T_m$.   The roots of Chebyshev polynomials are often
used as interpolation nodes instead.  There are few meaningful differences
between these two choices, although
(\ref{preliminaries:chebyshev:nodes}) 
includes the endpoints $\pm 1$, which is convenient when solving
boundary value problems for ordinary differential equations.
\end{remark}

\label{section:preliminaries:chebyshev}
\end{subsection}

\label{section:preliminaries}
\end{section}

\begin{section}{Analytical apparatus}

Here we develop the analytic apparatus used in Section~\ref{section:algorithm}
to design an algorithm for the numerical solution of second order linear ordinary
differential equations of the form (\ref{introduction:second_order}) whose
running time is independent of the parameter $\lambda$.

In Section~\ref{section:kummer}, we reformulate Kummer's equation as a nonlinear
integral equation in preparation for a statement of the main theorem
of \cite{Bremer-Rokhlin}.  
This is done in Section~\ref{section:nonoscillatory},
and several consequences of this result are discussed there.   
In Section~\ref{section:bound}, we develop a theorem which bounds
the restriction of the  solution $r$ of the logarithm form of Kummer's equation 
to an interval of the form $(-\infty,x_0]$
under the assumption that the coefficient $q$ is nearly equal to $1$ there.
This result is recorded as Theorem~\ref{bound:theorem5}.
In Section~\ref{section:continuity}, we use standard techniques
from the theory of ordinary differential equations in order
to bound the change in the solution of 
the logarithm form of Kummer's equation when the initial
conditions and coefficient $q$ are perturbed.  

\begin{subsection}{Integral equation formulation of Kummer's equation}

In this section, we reformulate Kummer's equation 
\begin{equation}
\left(\alpha'(t)\right)^2 = \lambda^2q(t) - \frac{1}{2}\frac{\alpha'''(t)}{\alpha'(t)}
+ \frac{3}{4} \left(\frac{\alpha''(t)}{\alpha'(t)}\right)^2
\label{kummer:kummers_equation}
\end{equation}
as a nonlinear integral equation in preparation for the statement
of the principal result of \cite{Bremer-Rokhlin}.    We assume that the function $q$ 
has been extended to the real line.

By letting
\begin{equation}
\alpha'(t) = \lambda \exp\left(\frac{r(t)}{2}\right)
\label{kummer:relation}
\end{equation}
in (\ref{kummer:kummers_equation}) we obtain 
\begin{equation}
r''(t) - \frac{1}{4}\left(r'(t)\right)^2 + 4 \lambda^2\left( \exp(r(t)) - q(t)\right) = 0,
\label{kummer:logarithm_form}
\end{equation}
which we refer to as the logarithm form of Kummer's equation.  
Representing the solution $r$ of (\ref{kummer:logarithm_form})
in the form
\begin{equation}
r(t) = \log(q(t)) + \delta(t)
\label{kummer:delta}
\end{equation} 
results in the equation
\begin{equation}
\delta''(t) - \frac{1}{2} \frac{q'(t)}{q(t)} \delta'(t) 
- \frac{1}{4} \left(\delta'(t)\right)^2  
+ 4\lambda^2 q(t) \left(\exp(\delta(t)-1)\right)
=
q(t) p(t),
\label{kummer:eq1}
\end{equation}
where the function $p$ is defined by  the formula
\begin{equation}
p(t) = 
\frac{1}{q(t)}
\left(
\frac{5}{4} \left(\frac{q'(t)}{q(t)}\right)^2
-
\frac{q''(t)}{q(t)}
\right).
\label{kummer:p}
\end{equation}
By expanding the exponential in (\ref{kummer:eq1})
in a power series and rearranging terms we obtain
\begin{equation}
\begin{aligned}
\delta''(t)  -  \frac{1}{2} \frac{q'(t)}{q(t)} \delta'(t) 
+ 4\lambda^2 q(t)\delta(t)
- \frac{1}{4}\left(\delta'(t)\right)^2
+ 4\lambda^2 q(t) 
\left(
\frac{\left(\delta(t)\right)^2}{2!} + 
\frac{\left(\delta(t)\right)^3}{3!} + 
\cdots\right)
= q(t)p(t).
\end{aligned}
\label{kummer:eq2}
\end{equation}
The change of variables
\begin{equation}
x(t) = \int_a^t \sqrt{q(u)}\ du
\label{kummer:kummer_change_of_vars}
\end{equation}
transforms (\ref{kummer:eq2}) into
\begin{equation}
\delta''(x) + 4\lambda^2 \delta(x) 
=
S\left[\delta\right](x) + p(x),
\label{kummer:deltaeq}
\end{equation}
where $S$ is the nonlinear differential operator defined via the formula
\begin{equation}
S\left[f\right](x) 
= \frac{\left(f'(x)\right)^2}{4}
- 4 \lambda^2\left(
\frac{\left(f(x)\right)^2}{2!} + 
\frac{\left(f(x)\right)^3}{3!} + 
\cdots
\right).
\label{kummer:S}
\end{equation}

We observe that the function $p(t)$ defined in formula
(\ref{kummer:p}) is related
to the Schwarzian derivative  (see Section~\ref{section:preliminaries:schwarzian_derivative})
$\{x,t\}$ of the function  $x$ defined in
(\ref{kummer:kummer_change_of_vars})
 via the formula 
\begin{equation}
 p(t) =  -\frac{2}{q(t)} \{x,t\}
=
-2 \left(\frac{dt}{dx}\right)^2 \{x,t\}.
\label{kummer:pscw}
\end{equation}
From (\ref{kummer:pscw}) and  Formula~(\ref{preliminaries:Schwarzian_derivative:change_of_vars})
in Section~\ref{section:preliminaries:schwarzian_derivative},
we see that
\begin{equation}
p(x) = 2 \{ t, x \};
\end{equation}
that is the function $p(x)$ is twice the Schwarzian derivative of 
the inverse of the function $x(t)$.

We also observe that the differential operator appearing on the
left-hand side of (\ref{kummer:deltaeq}) 
is the constant coefficient Helmholtz equation.
In order to exploit this observation,
we define the operator $T$ for functions $f \in \Sr$
via the formula
\begin{equation}
T\left[f\right](x) 
= \frac{1}{4\lambda} \int_{-\infty}^\infty\sin \left(2\lambda \left|x-y\right|\right) f(y)\ dy
\ \ \ \mbox{for all} \ \ x\in\mathbb{R}.
\label{kummer:T}
\end{equation}
According to Theorem~\ref{preliminaries:helmholtz:theorem1}, 
$T\left[f\right]$ is the unique solution of the ordinary differential equation
\begin{equation}
 y''(x) + 4 \lambda^2 y(x)  = f(x)
\label{kummer:helmholtz}
\end{equation}
such that
\begin{equation}
\widehat{T\left[f\right]}(\xi) = 
\frac{\widehat{f}(\xi)}{4\lambda^2-\xi^2}.
\label{kummer:That}
\end{equation}
Consequently, introducing the representation
\begin{equation}
\delta(x) = T\left[\sigma\right](x)
\end{equation}
into (\ref{kummer:deltaeq}) results in the nonlinear
integral equation
\begin{equation}
\sigma(x) = S \left[ T \left[ \sigma \right]\right](x)  + p(x).
\label{kummer:inteq}
\end{equation}

\label{section:kummer}
\end{subsection}

%
%

\begin{subsection}{Nonoscillatory solutions of Kummer's equation}

Equation (\ref{kummer:inteq}) does not admit solutions for all functions $p$.  
However,  according to the following result, which appears
as Theorem~12 in  \cite{Bremer-Rokhlin},
if the function $p$ is nonoscillatory
then there exists a function $\nu$ of small magnitude
 such that the nonlinear integral equation
\begin{equation}
\sigma(x) = S\left[T\left[\sigma\right]\right] + p(x) + \nu(x)
\label{nonoscillatory:perturbedinteq}
\end{equation}
admits a solution $\sigma$ which is also nonoscillatory.

%
%
\vskip 1em
\begin{theorem}
Suppose that  $q \in C^\infty\left(\mathbb{R}\right)$ is strictly positive,
that $x(t)$ is defined by the formula
\begin{equation}
x(t) = \int_0^t \sqrt{q(u)}\ du,
\label{main_theorem:x}
\end{equation}
and that the function $p$ defined via the formula
\begin{equation}
p(x) = 2\{t,x\}
\end{equation}
is an element of $\Sr$. Suppose further that there exist positive real numbers
$\lambda$, $\Gamma$ and $\mu$ such that
\begin{equation}
  \lambda \geq  4\max\left\{\frac{1}{\mu},\Gamma\right\}
\label{main_theorem:lambda}
\end{equation}
and
\begin{equation}
\left|\widehat{p}(\xi)\right| \leq 
\Gamma 
\exp\left(-\mu\left|\xi\right|\right)
\ \ \ \mbox{for all}\ \ \xi\in\mathbb{R}.
\end{equation}
Then there exist functions $\sigma$ and $\nu$  in $\Sr$ 
such that $\sigma$ is a solution of the nonlinear integral equation
\begin{equation}
\sigma(x) = S\left[T\left[\sigma\right]\right](x) + p(x) + \nu(x),
\ \ \ \mbox{for all}\ \ x\in\mathbb{R},
\label{main_theorem:inteq}
\end{equation}
\begin{equation}
\left|\widehat{\sigma}(\xi)\right| \leq
\frac{3}{2\Gamma}
\exp\left(- \mu|\xi|\right) 
\ \ \ \mbox{for all}\ \ \left|\xi\right| \leq \sqrt{2}\lambda,
\label{main_theorem:sigma1}
\end{equation}
\begin{equation}
\widehat{\sigma}(\xi) = 0 
\ \ \ \mbox{for all}\ \ \left|\xi\right| > \sqrt{2}\lambda,
\label{main_theorem:sigma2}
\end{equation}
and
\begin{equation}
\|\nu\|_\infty \leq 
\frac{\Gamma}{\mu}
 \exp\left(-\mu\lambda\right).
\label{main_theorem:nu}
\end{equation}
\label{main_theorem}
\end{theorem}

Suppose that $\sigma$ and $\nu$ are the functions obtained by invoking
Theorem~\ref{main_theorem}, and that 
$x(t)$ is the function defined by the formula
\begin{equation}
x(t) = \int_a^t \sqrt{q(u)}\ du.
\end{equation}
We define  $\delta$ by the formula
\begin{equation}
\delta(x) = T\left[\sigma\right](x),
\end{equation}
$r$ by the formula
\begin{equation}
r(t) = \log(q(t)) + \delta(x(t)),
\label{nonoscillatory:r}
\end{equation}
and $\alpha$ by the formula
\begin{equation}
\alpha(t) = \lambda \int_a^t  \exp\left(\frac{r(u)}{2}\right)\ du.
\label{nonoscillatory:alpha}
\end{equation}
From the discussion in Section~\ref{section:kummer}, we conclude that
$\delta(x)$ is a solution of the nonlinear differential equation
\begin{equation}
\delta''(x) + 4\lambda^2 \delta(x) = S\left[\delta\right](x) + p(x) + \nu(x)
\ \ \ \mbox{for all} \ \ x\in\mathbb{R},
\label{nonoscillatory:deltax}
\end{equation}
that $\delta(x(t))$ is a solution of the nonlinear differential equation
\begin{equation}
\delta''(t) - \frac{1}{2}\frac{q'(t)}{q(t)}\delta'(t) 
-\frac{1}{4}\left(\delta'(t)\right)^2
+ 4\lambda^2 
q(t) \left(\exp(\delta(t))-1\right) 
=
q(t)\left( p(t) + \nu(t)\right) 
\ \ \ \mbox{for all}\ \ \ t\in\mathbb{R},
\end{equation}
that $r(t)$ is a solution of the nonlinear differential equation
\begin{equation}
r''(t) - \frac{1}{4}(r'(t))^2 + 4\lambda^2  \left(\exp(r(t))-q(t)\right) =  q(t) \nu(t)
\ \ \ \mbox{for all}\  \ t\in\mathbb{R},
\label{nonoscillatory:kummer_logarithm}
\end{equation}
and that $\alpha$ is a solution of the nonlinear differential equation
\begin{equation}
\left(\alpha'(t)\right)^2 = \lambda^2 
\left(\frac{\nu(t)}{4\lambda^2}+1\right)q(t) - \frac{1}{2}\frac{\alpha'''(t)}{\alpha'(t)}
+ \frac{3}{4} \left(\frac{\alpha''(t)}{\alpha'(t)}\right)^2
\ \ \ \mbox{for all}\ \ t \in\mathbb{R}.
\label{nonoscillatory:kummer}
\end{equation}
From (\ref{nonoscillatory:kummer}), we see that $\alpha$ is a phase function
for the second order linear  ordinary differential equation
\begin{equation}
y''(t) + \lambda^2 \left(1+\frac{\nu(t)}{4\lambda^2}\right)q(t)y(t) = 0
\ \ \ \mbox{for all}\ \ a \leq t \leq b.
\label{nonoscillatory:perturbed_equation}
\end{equation}
The following result,
which appears as Theorem~14 in \cite{Bremer-Rokhlin},
bounds the order of magnitude of the  difference between solutions of 
(\ref{nonoscillatory:perturbed_equation}) and those of 
(\ref{introduction:second_order}).

\vskip 1em
\begin{theorem}
Suppose that the hypotheses of Theorem~\ref{main_theorem} are satisfied,
that $\sigma$ and $\nu$ are the functions obtained by invoking
it.  Suppose also that $\alpha$ is defined as in (\ref{nonoscillatory:alpha}),
and that $u$, $v$ are the functions defined via the formulas
\begin{equation}
u(t) = \frac{\cos(\alpha(t))}{\sqrt{\alpha'(t)}}
\label{overview:u}
\end{equation}
and
\begin{equation}
v(t) = \frac{\sin( \alpha(t))}{\sqrt{\alpha'(t)}}.
\label{overview:v}
\end{equation}
Then there exist a constant $C$ and a basis $\{\tilde{u},\tilde{v}\}$ in the space
of solutions of (\ref{introduction:second_order}) such that 
\begin{equation}
\left|u(t) - \tilde{u}(t) \right| \leq \frac{C}{\lambda} \exp\left(-\mu\lambda\right)
\ \ \ \mbox{for all}\  \ a \leq t \leq b
\end{equation}
and
\begin{equation}
\left|v(t) - \tilde{v}(t) \right| \leq \frac{C}{\lambda} \exp\left(-\mu\lambda\right)
\ \ \ \mbox{for all}\  \ a \leq t \leq b.
\end{equation}
The constant $C$ depends on the coefficient $q$ appearing in (\ref{introduction:second_order}),
but not on the parameter $\lambda$.
\label{nonoscillatory:theorem2}
\end{theorem}

\label{section:nonoscillatory}
\end{subsection}

\begin{subsection}{A bound on $\delta$ in the event that $p$ is small in magnitude}

In this section, we bound the solution
$\delta$  of the nonlinear differential equation (\ref{nonoscillatory:deltax}) and its derivative
under an assumption on the function $p$ appearing  in
(\ref{nonoscillatory:deltax}).
More specifically, we show that when the function $p$ 
 is of sufficiently small magnitude on an interval of the form $(-\infty,x_0]$ 
and $\lambda$ is sufficiently large,
the restrictions of $\delta$ and $\delta'$ to $(-\infty,x_0]$ are on the order
of 
\begin{equation}
\exp\left(-\frac{\lambda \mu}{2}\right).
\end{equation}

We proceed by perturbing the parameter $\lambda$
in the linear differential operator appearing on the left-hand side
of Equation~(\ref{nonoscillatory:deltax}) by an imaginary constant $i\eta$
of small magnitude. This results in the nonlinear differential equation
\begin{equation}
\delta_\eta''(x)  + 4(\lambda+i\eta)^2 \delta_\eta(x) =  S\left[\delta_\eta\right](x) + p(x) + \nu(x).
\label{bound:delta_eta}
\end{equation}
We then develop estimates on the magnitudes of 
the restrictions of a
solution $\delta_\eta$ of (\ref{bound:delta_eta})
and its derivative $\delta_\eta'$ to the interval $(-\infty,x_0]$.
We use these estimates to bound the restriction of
 $\delta$ and its derivative $\delta'$ to
the interval $(-\infty,x_0]$.
The advantage of  (\ref{bound:delta_eta}) over the nonlinear differential equation
\begin{equation}
\delta''(x)  + 4\lambda^2 \delta(x) =  S\left[\delta\right](x) + p(x) + \nu(x)
\label{bound:delta}
\end{equation}
defining $\delta$ is that the fundamental solution 
\begin{equation}
\frac{1}{4(\lambda+i\eta)i}\exp\left(2 (\lambda + i \eta) i |x|\right)
\label{bound:green1}
\end{equation}
of (\ref{bound:delta_eta}) associated with the Fourier transform is an element of $\Lp{1}$.  
This is in contrast to  the fundamental solution 
\begin{equation}
\frac{1}{4 \lambda}\sin\left(2 \lambda |x|\right)
\end{equation}
for the equation (\ref{bound:delta}) associated with the Fourier transform,
which is not absolutely integrable.

We begin by defining, for any positive real number $\eta$,
the operator $T_\eta$ for functions $f \in \Sr$ via the formula
\begin{equation}
T_\eta\left[f\right](x) = 
\frac{1}{4 \left(\lambda + i\eta\right)}
\int_{-\infty}^\infty \exp\left(2 (\lambda + i \eta) \left|x-y\right|\right) f(y)\ dy.
\label{bound:Teta}
\end{equation}
According to Theorem~\ref{preliminaries:helmholtz:theorem1},
$T_\eta\left[f\right]$ is the unique solution of the equation
\begin{equation}
y''(t) + 4 (\lambda+i\eta)^2 y(t) = f(t)
\end{equation}
such that
\begin{equation}
\widehat{T_\eta\left[f\right]}(\xi) = 
\frac{\widehat{f}(\xi)}{4(\lambda+i\eta)^2 - \xi^2}
\ \ \ \mbox{for all}\ \ \xi \in \mathbb{R}.
\end{equation}
The following theorem bounds  the difference
between $T_\eta\left[\sigma\right]$ and $T\left[\sigma\right]$ 
in the event that $\sigma$ satisfies the
conclusions of Theorem~\ref{main_theorem}.

%
%

\vskip 1em
\begin{theorem}
Suppose that $\sigma \in \Sr$, 
that there exist  positive real numbers $\mu$, $\Gamma$ and $\lambda$ such that
\begin{equation}
\left|\widehat{\sigma}(\xi)\right|
\leq 
\frac{3\Gamma}{2} \exp\left(- \mu \left|\xi\right|\right)
\ \ \ \mbox{for all} \ \ \left|\xi\right| < \sqrt{2}\lambda,
\label{bound:theorem1:assumption1}
\end{equation}
and that
\begin{equation}
\widehat{\sigma}(\xi) = 0
\ \ \ \mbox{for all} \ \ \left|\xi\right| \geq \sqrt{2}\lambda.
\label{bound:theorem1:assumption2}
\end{equation}
Suppose also that $\eta$ is a positive real number such that
\begin{equation}
2 \eta \leq \lambda.
\end{equation}
Suppose further that $\delta$ is defined via the formula
\begin{equation}
\delta(x) = T\left[\sigma\right](x) = 
\frac{1}{4\lambda} \int_{-\infty}^\infty 
\sin\left(2 \lambda   \left|x-y\right|\right)
\sigma(y)\ dy,
\label{bound:theorem1:delta}
\end{equation}
and that $\delta_\eta$ is defined via the formula
\begin{equation}
\delta_\eta(x) = 
T_\eta\left[\sigma\right](x) = 
\frac{1}{4 \left(\lambda + i\eta\right) i } \int_{-\infty}^\infty 
\exp\left(2  \left(\lambda + i \eta\right) i \left|x-y\right|\right)
\sigma(y)\ dy.
\label{bound:theorem1:deltaeta}
\end{equation}
Then 
\begin{equation}
\lim_{\left|x\right|\to\infty}  \left|\delta_\eta(x)\right| + \left|\delta_\eta'(x)\right| = 0,
\label{bound:theorem1:conclusion1}
\end{equation}
\begin{equation}
\left|\delta(x)\right|
\leq 
\frac{3\Gamma}{4\pi \mu\lambda^2}
\ \ \ \mbox{for all} \ \ x \in \mathbb{R},
\label{bound:theorem1:conclusion2}
\end{equation}
\begin{equation}
\left|\delta'(x)\right| \leq
\frac{3\Gamma }{4\pi \mu^2 \lambda^2},
\ \ \ \mbox{for all} \ \ x \in \mathbb{R},
\label{bound:theorem1:conclusion3}
\end{equation}
\begin{equation}
\left| \delta(x) - \delta_\eta(x)\right| \leq 
\frac{3\Gamma \eta}{\pi \mu \lambda^3}
\ \ \ \mbox{for all} \ \ x \in \mathbb{R},
\label{bound:theorem1:conclusion4}
\end{equation}
and
\begin{equation}
\left| \delta'(x) - \delta_\eta'(x)\right| \leq 
\frac{3\Gamma  \eta}{\pi \mu^2\lambda^3}.
\ \ \ \mbox{for all} \ \ x \in \mathbb{R}.
\label{bound:theorem1:conclusion5}
\end{equation}
\label{bound:theorem1}
\end{theorem}

\begin{proof}
We observe that functions
\begin{equation}
\frac{\widehat{\sigma}(\xi)}{4(\lambda+i\eta)^2-\xi^2}
\label{bound:theorem1:1}
\end{equation}
and 
\begin{equation}
\frac{i\xi \widehat{\sigma}(\xi)}{4(\lambda+i\eta)^2-\xi^2},
\label{bound:theorem1:2}
\end{equation}
are elements of $C_c^\infty\left(\mathbb{R}\right)$.
Among other things, this implies that the inverse Fourier transforms of 
(\ref{bound:theorem1:1}) and (\ref{bound:theorem1:2}),
which are  $\delta_\eta$ and $\delta_\eta'$, respectively, are elements of $\Sr$.
The conclusion (\ref{bound:theorem1:conclusion1}) follows immediately
from this observation.

An elementary calculation shows that
\begin{equation}
\left|
\frac{1}{4(\lambda+i\eta)^2-\xi^2}
-
\frac{1}{4\lambda^2-\xi^2}
\right|
= 
\left|\frac{4\eta}{4\lambda^2-\xi^2}\right|
\sqrt{
\frac{4\lambda^2 + \eta^2}{(4\lambda^2-\xi^2-4\eta^2)^2+64\eta^2\lambda^2}
}.
\label{bound:theorem1:5}
\end{equation}
We observe that
\begin{equation}
\left|\frac{4\eta}{4\lambda^2-\xi^2}\right|
\leq 
 \frac{2 \eta}{\lambda^2}
\label{bound:theorem1:6}
\end{equation}
for all $\eta > 0$ and $|\xi| \leq \sqrt{2}\lambda$. 
Moreover, 
\begin{equation}
4\lambda^2 - 4\eta^2 - \xi^2
\geq
2\lambda^2 - 4 \eta^2
\geq 0 
\label{bound:theorem1:7}
\end{equation}
for all $|\xi| \leq \sqrt{2}\lambda$ and $\lambda \geq 2\eta$.   
It follows from 
(\ref{bound:theorem1:7}) that
\begin{equation}
\begin{aligned}
\frac{4\lambda^2 + \eta^2}{(4\lambda^2-\xi^2-4\eta^2)^2+64\eta^2\lambda^2}
&\leq 
\frac{ 4\lambda^2 + \eta^2}{(2\lambda^2 - 4\eta^2)^2+64\eta^2\lambda^2}
\\
&=
\frac{4\lambda^2 + \eta^2}{4\lambda^4 + 48 \lambda^2 \eta^2 +16\eta^4}
\\
&\leq
\frac{1}{\lambda^2}
\frac{4\lambda^2 + \eta^2}{4\lambda^2 + 48 \eta^2}\\
&\leq
\frac{1}{\lambda^2}
\end{aligned}
\label{bound:theorem1:8}
\end{equation}
for all $ 0 < 2\eta \leq \lambda$, and $|\xi|\leq \sqrt{2}\lambda$.
We insert (\ref{bound:theorem1:8})
and (\ref{bound:theorem1:6})
 into (\ref{bound:theorem1:5})
in order to conclude that
\begin{equation}
\left|
\frac{1}{4(\lambda+i\eta)^2-\xi^2}
-
\frac{1}{4\lambda^2-\xi^2}
\right|
\leq
\frac{2 \eta}{\lambda^3}
\label{bound:theorem1:9}
\end{equation}
for all $0 < 2\eta \leq \lambda$ and $|\xi|< \sqrt{2}\lambda$.
From Theorem~\ref{preliminaries:helmholtz:theorem1} and 
(\ref{bound:theorem1:delta})  we conclude that
\begin{equation}
\begin{aligned}
\left\|\delta\right\|_\infty
\leq 
\frac{1}{2\pi} \left\|\widehat{\delta}\right\|_1
\leq
\frac{1}{2\pi} 
\int_{-\infty}^\infty \left|\frac{\widehat{\sigma}(\xi)}{4\lambda^2-\xi^2}\right|\ d\xi,
\end{aligned}
\label{bound:theorem1:10}
\end{equation}
and that
\begin{equation}
\begin{aligned}
\left\|\delta'\right\|_\infty
\leq 
\frac{1}{2\pi} \left\|\widehat{\delta'}\right\|_1
\leq
\frac{1}{2\pi} 
\int_{-\infty}^\infty \left|\frac{i\xi\widehat{\sigma}(\xi)}{4\lambda^2-\xi^2}\right|\ d\xi.
\end{aligned}
\label{bound:theorem1:11}
\end{equation}
We insert  (\ref{bound:theorem1:assumption1})
and (\ref{bound:theorem1:assumption2})
into (\ref{bound:theorem1:10}) in order to conclude that
\begin{equation}
\begin{aligned}
\left\|\delta\right\|_\infty 
\leq 
\frac{3\Gamma}{4\pi}
 \int_{-\sqrt{2}\lambda}^{\sqrt{2}\lambda} 
\left|\frac{1}{4\lambda^2-\xi^2}\right|
\exp\left(-\mu \left|\xi\right|\right)
\ d\xi
\leq 
\frac{3\Gamma}{8\pi\lambda^2} \int_{-\sqrt{2}\lambda}^{\sqrt{2}\lambda} 
\exp\left(-\mu \left|\xi\right|\right)
\ d\xi
\leq
\frac{3\Gamma }{4\pi\mu\lambda^2},
\end{aligned}
\label{bound:theorem1:12}
\end{equation}
which is (\ref{bound:theorem1:conclusion2}).
By inserting
(\ref{bound:theorem1:assumption1}) and  (\ref{bound:theorem1:assumption2}) 
into (\ref{bound:theorem1:11}) we obtain
\begin{equation}
\begin{aligned}
\left\|\delta'\right\|_\infty 
&\leq 
\frac{3\Gamma}{4\pi} \int_{-\sqrt{2}\lambda}^{\sqrt{2}\lambda} 
\left|\frac{i\xi}{4\lambda^2-\xi^2}\right|
\exp\left(-\mu \left|\xi\right|\right)
\ d\xi
\leq 
\frac{3\Gamma}{8 \pi \lambda^2} \int_{-\sqrt{2}\lambda}^{\sqrt{2}\lambda} 
\left|\xi\right|\exp\left(-\mu \left|\xi\right|\right)
\ d\xi
\leq
\frac{3\Gamma }{4\pi \mu^2 \lambda^2},
\end{aligned}
\label{bound:theorem1:13}
\end{equation}
which is (\ref{bound:theorem1:conclusion3}).
Next, we observe that
\begin{equation}
\begin{aligned}
\left\|\delta - \delta_\eta\right\|_\infty 
\leq 
\frac{1}{2\pi} 
\left\|\widehat{\delta} - \widehat{\delta_\eta}\right\|_1
\leq 
\frac{1}{2\pi}
\int_{-\infty}^{\infty}
\left|
\frac{\widehat{\sigma}(\xi)}{4\lambda^2-\xi^2} - 
\frac{\widehat{\sigma}(\xi)}{4(\lambda+i\eta)^2-\xi^2}\right|
\ d\xi,
\end{aligned}
\label{bound:theorem1:14}
\end{equation}
and that
\begin{equation}
\begin{aligned}
\left\|\delta' - \delta_\eta'\right\|_\infty 
\leq 
\frac{1}{2\pi} 
\left\|i\xi\widehat{\delta} - i\xi\widehat{\delta_\eta}\right\|_1
\leq 
\frac{1}{2\pi}
\int_{-\infty}^{\infty}
\left|
\frac{i\xi\widehat{\sigma}(\xi)}{4\lambda^2-\xi^2} - 
\frac{i\xi\widehat{\sigma}(\xi)}{4(\lambda+i\eta)^2-\xi^2}\right|
\ d\xi.
\end{aligned}
\label{bound:theorem1:15}
\end{equation}
We insert  (\ref{bound:theorem1:assumption1}),
(\ref{bound:theorem1:assumption2})  
and (\ref{bound:theorem1:9}) into 
(\ref{bound:theorem1:14}) in order to establish that
\begin{equation}
\begin{aligned}
\left\|\delta - \delta_\eta\right\|_\infty 
\leq
\frac{3 \Gamma\eta}{2\pi\lambda^3}
\int_{-\sqrt{2}\lambda}^{\sqrt{2}\lambda} 
\exp\left(-\mu \left|\xi\right|\right)
\ d\xi
\leq 
\frac{3\Gamma \eta}{\pi \mu \lambda^3},
\end{aligned}
\label{bound:theorem1:16}
\end{equation}
for all $0 <\eta < 2\lambda$,
which is the conclusion (\ref{bound:theorem1:conclusion4}).
Finally, we combine  (\ref{bound:theorem1:assumption1}),
(\ref{bound:theorem1:assumption2})  
and (\ref{bound:theorem1:9})
with
(\ref{bound:theorem1:15}) in order to conclude that
\begin{equation}
\begin{aligned}
\left\|\delta' - \delta_\eta'\right\|_\infty 
\leq 
\frac{3\Gamma  \eta}{2\pi\lambda^3}
\int_{-\sqrt{2}\lambda}^{\sqrt{2}\lambda} 
\left|\xi\right|\exp\left(- \mu \left|\xi\right|\right)
\ d\xi
\leq 
\frac{3\Gamma  \eta}{\pi \mu^2\lambda^3},
\end{aligned}
\label{bound:theorem1:17}
\end{equation}
for all $0 < 2 \eta < \lambda$, 
which establishes (\ref{bound:theorem1:conclusion5}).
\end{proof}

%
%

We now use the conclusions of Theorem~\ref{bound:theorem1}
in order to bound the magnitude of $S\left[\delta\right]$, where $S$
is the nonlinear differential operator defined in (\ref{kummer:S}), 
in terms  of the solution $\delta_\eta$ of the complexified equation
and its derivative.

\vskip 1em
\begin{theorem}
Suppose that the hypotheses of Theorem~\ref{bound:theorem1}
are satisfied, and that
\begin{equation}
\lambda \geq 4 \max\left\{\Gamma, \frac{1}{\mu}\right\}.
\label{bound:theorem2:lambda}
\end{equation}
Suppose further  that $S$ is the nonlinear differential operator defined via 
(\ref{kummer:S}).    Then
\begin{equation}
\left|S\left[\delta\right](x)\right|
\leq 
\frac{\eta^2}{68}
+
\frac{\left|\delta_\eta'(x)\right|^2 }{2}
+ \frac{102 \lambda^2}{25} \left|\delta_\eta(x)\right|^2
\ \ \ \mbox{for all}\ \ x\in\mathbb{R}.
\label{complexifications:theorem2:conclusion}
\end{equation}
\label{bound:theorem2}
\end{theorem}

\begin{proof}
We define the function $\tau$ via the formula
\begin{equation}
\tau(x) = \delta(x) - \delta_\eta(x)
\label{bound:theorem2:1}
\end{equation}
so that $\delta = \delta_\eta + \tau_\eta$.  
We invoke Theorem~\ref{bound:theorem1}
and exploit the assumption (\ref{bound:theorem2:lambda})
in order to conclude that
\begin{equation}
\left|\tau(x)\right| \leq 
\frac{3\eta}{16 \pi \lambda},
\ \ \ \mbox{for all} \ \ x \in\mathbb{R}
\label{bound:theorem2:2}
\end{equation}
\begin{equation}
\left|\tau'(x)\right| \leq 
\frac{3\eta}{ 64 \pi}
\ \ \ \mbox{for all} \ \ x \in\mathbb{R}.
\label{bound:theorem2:3}
\end{equation}
and
\begin{equation}
\left|\delta(x)\right| \leq 
\frac{3}{64 \pi}
\ \ \ \mbox{for all} \ \ x \in\mathbb{R}.
\label{bound:theorem2:3.5}
\end{equation}
From the definition (\ref{kummer:S})
of $S$  and the triangle inequality we conclude that
\begin{equation}
\begin{aligned}
\left|S\left[\delta\right](x)\right| 
&\leq
\frac{\left|\left(\delta'(x)\right)^2\right|}{4}
+ 4\lambda^2
\left|
\exp(\delta(x)) - \delta(x) - 1
\right|
\end{aligned}
\label{bound:theorem2:4}
\end{equation}
for all $x \in \mathbb{R}$.  We insert the inequality
\begin{equation}
\left|\exp(x) -x - 1\right| \leq  \frac{\left|x\right|^2}{2} \exp(\left|x\right|)
\ \ \ \mbox{for all}\ \ \ x \in\mathbb{R}
\label{bound:theorem2:5}
\end{equation}
into (\ref{bound:theorem2:4}) in order to conclude that 
\begin{equation}
\begin{aligned}
\left|S\left[\delta\right](x)\right| 
&\leq
\frac{\left|\delta'(x)\right|^2}{4}
+ 2\lambda^2
\exp\left(\left|\delta(x)\right|\right)
\left|\delta(x)\right|^2
\ \ \ \mbox{for all}\ \ x\in\mathbb{R}.
\end{aligned}
\label{bound:theorem2:6}
\end{equation}
From (\ref{bound:theorem2:3.5}) we obtain
\begin{equation}
\exp\left(\left|\delta(x)\right|\right) 
\leq 
\exp\left(\frac{3}{64\pi}\right) 
\leq 
\frac{102}{100}
\ \ \ \mbox{for all}\ \ x\in\mathbb{R}.
\label{bound:theorem2:7}
\end{equation}
We insert (\ref{bound:theorem2:7}) 
into (\ref{bound:theorem2:6}) in order to conclude that
\begin{equation}
\begin{aligned}
\left|S\left[\delta\right](x)\right| 
&\leq
\frac{\left|\delta'(x)\right|^2}{4}
+ \frac{204 \lambda^2}{100}
\left|\delta(x)\right|^2
\ \ \ \mbox{for all}\ \ x\in\mathbb{R}.
\end{aligned}
\label{bound:theorem2:8}
\end{equation}
We combine (\ref{bound:theorem2:1}) with (\ref{bound:theorem2:8}) and  the fact that
\begin{equation}
(\left|x + y\right|)^2 \leq
(\left|x\right|+\left|y\right|)^2
\leq
2 \left(\left|x\right|^2 +\left|y\right|^2\right)
\ \ \ \mbox{for all} \ \ x,y \in \mathbb{R}
\label{bound:theorem2:9}
\end{equation}
in order to conclude that
\begin{equation}
\begin{aligned}
\left|S\left[\delta\right](x)\right| 
&\leq
\frac{\left|\delta_\eta'(x)\right|^2 + \left|\tau'(x)\right|^2}{2}
+ \frac{408 \lambda^2 }{100}
 \left(\left|\delta_\eta(x)\right|^2 + \left|\tau(x)\right|^2 \right)
\end{aligned}
\label{bound:theorem2:10}
\end{equation}
for all $x \in \mathbb{R}$.  Next we insert (\ref{bound:theorem2:2}) and
(\ref{bound:theorem2:3}) into (\ref{bound:theorem2:10}) 
in order to conclude that
\begin{equation}
\begin{aligned}
\left|S\left[\delta\right](x)\right| 
&\leq
\frac{\eta^2}{68}
+
\frac{\left|\delta_\eta'(x)\right|^2 }{2}
+ \frac{102 \lambda^2}{25} \left|\delta_\eta(x)\right|^2
\\
\end{aligned}
\end{equation}
for all $x \in \mathbb{R}$, 
which is the conclusion of the theorem.
\end{proof}

The following theorem bounds the magnitude of the Fourier transform of the product of
$\sigma$ with a decaying exponential function at the points $\pm 2 \lambda$.

%
%

\vskip 1em
\begin{theorem}
Suppose that the hypotheses of  Theorem~\ref{bound:theorem1}
are satisfied, and that $x$ is a real number.
Then
\begin{equation}
\left| 
\int_{-\infty}^\infty
\exp(\pm 2\lambda i y) 
\exp\left(-2\eta \left|x-y\right|\right)
\sigma(y)\ dy
\right|
\leq 
\frac{3 \Gamma }{2}
\exp\left(- \frac{\lambda \mu}{2}\right).
\label{bound:theorem3:conclusion}
\end{equation}
\label{bound:theorem3}
\end{theorem}
\begin{proof}
For any $x \in \mathbb{R}$, we define the function $g_x$ via the formula
\begin{equation}
g_x(y) = \exp\left(-2 \eta \left|x-y\right|\right).
\label{bound:theorem3:1}
\end{equation}
We observe that
\begin{equation}
\widehat{g_x}(\xi) = \frac{4 \eta \exp(-ix\xi)}{4\eta^2 + \xi^2}.
\label{bound:theorem3:2}
\end{equation}
Consequently, 
\begin{equation}
\begin{aligned}
\widehat{ \sigma \cdot g_x }(\pm 2\lambda) 
&=
\frac{1}{2\pi} \int_{-\infty}^{\infty}
\widehat{\sigma}(\pm 2\lambda - \xi)\ \widehat{g_x}(\xi)\ d\xi
= 
\frac{1}{2\pi} \int_{-\infty}^{\infty}
\widehat{\sigma}(\pm 2\lambda - \xi)
\frac{ 4 \eta \exp(-i x\xi)}{4\eta^2 + \xi^2}\ d\xi.
\end{aligned}
\label{bound:theorem3:3}
\end{equation}
We insert (\ref{bound:theorem1:assumption1})
and (\ref{bound:theorem1:assumption2}) into (\ref{bound:theorem3:3})
in order to obtain
\begin{equation}
\begin{aligned}
\left|\widehat{g_x \cdot \sigma}(\pm 2\lambda) \right|
&\leq
\frac{3\Gamma}{4\pi} \int_{-\sqrt{2}\lambda}^{\sqrt{2}\lambda}
\exp\left(- \left|\pm 2\lambda - \xi\right| \mu\right)
\frac{ 4 \eta}{4\eta^2 + \xi^2}
d\xi
\\
&\leq
\frac{3\Gamma }{4\pi}
\exp\left(- \left(2\lambda - \sqrt{2}\lambda\right)\mu\right) 
\int_{-\infty}^{\infty}
\frac{ 4 \eta}{4\eta^2 + \xi^2}
d\xi
\\
&\leq
\frac{3 \Gamma}{2}
\exp\left(- \frac{\lambda \mu}{2}\right).
\end{aligned}
\label{bound:theorem3:4}
\end{equation}
which is (\ref{bound:theorem3:conclusion}).  In (\ref{bound:theorem3:4}),
we used the (easily verifiable) fact that
\begin{equation}
\int_{-\infty}^{\infty} \frac{ 4 \eta}{4\eta^2 + \xi^2} d\xi = 2\pi
\ \ \ \mbox{for all}\ \ \eta > 0.
\end{equation}
\end{proof}

We now combine Theorems~\ref{bound:theorem1}, \ref{bound:theorem2}
and \ref{bound:theorem3} to develop  a  bound on the solution 
$\delta_\eta$ of the complexified equation
\begin{equation}
\delta_\eta''(x) + 4(\lambda+i\eta)^2 \delta_\eta(x) = \sigma(x) = 
S\left[\delta\right](x) + p(x) + \nu(x)
\end{equation}
and its derivative on the interval $(-\infty,x_0]$ under the assumption
that the function $p$ is of small magnitude there.

%
%

\vskip 1em
\begin{theorem}
Suppose that the hypotheses of Theorem~\ref{main_theorem} are satisfied, and that
$\sigma$ and $\nu$ are the functions obtained by invoking it.
Suppose further that $\eta >0$ and $x_0$ are real numbers, that
\begin{equation}
\lambda \geq \max\left\{
\frac{4}{\mu}, 
4\Gamma,
2\eta,
\frac{2}{\mu} \log\left(\frac{24\Gamma}{\eta}\right),
\frac{1}{\mu} \log\left(\frac{16\Gamma}{\mu \eta^2}\right)
\right\},
\label{bound:theorem4:lambda}
\end{equation}
and that
\begin{equation}
\left| p(x)\right| \leq \frac{ \eta^2}{16}
\ \ \ \mbox{for all}\ \ x \leq x_0.
\label{bound:theorem4:pandv}
\end{equation}
Suppose also that $\delta_\eta$ is defined via the formula
\begin{equation}
\delta_\eta(x) = T_\eta\left[\sigma\right](x).
\label{bound:theorem4:delta_eta}
\end{equation}
Then 
\begin{equation}
\left|\delta_\eta(x) \right| \leq \frac{\eta}{4 \lambda}
\label{bound:theorem4:conclusion1}
\end{equation}
and
\begin{equation}
\left|\delta'_\eta(x) \right| \leq   \frac{\eta}{4}
\label{bound:theorem4:conclusion2}
\end{equation}
for all $x \leq x_0$.

\label{bound:theorem4}
\end{theorem}
\begin{proof}

From the definition (\ref{bound:Teta}) of $T_\eta$ and 
(\ref{bound:theorem4:delta_eta}) we conclude that
\begin{equation}
\begin{aligned}
\delta_\eta (x) 
&=
\frac{1}{4(\lambda+i\eta)i} \int_{-\infty}^{\infty}
\exp\left(2(\lambda+i\eta)i\left|x-y\right|\right)\ \sigma(y)\ dy
\\
&=
\frac{\exp(2\lambda i x)}{4(\lambda+i\eta)i} \int_{-\infty}^{x}
\exp\left(-2\eta \left|x-y\right|\right) \exp(-2\lambda iy)
\ \sigma(y)\ dy
\\
&+
\frac{\exp(-2\lambda i x)}{4(\lambda+i\eta)i} \int_{x}^{\infty}
\exp(-2 \eta \left|x-y\right|) 
\exp(2\lambda i y)
\ \sigma(y)\ dy
\end{aligned}
\label{bound:theorem4:1}
\end{equation}
for all $x \in \mathbb{R}$.  By differentiating (\ref{bound:theorem4:1})
we see that
\begin{equation}
\begin{aligned}
\delta_\eta' (x) &= 
\frac{1}{2} \int_{-\infty}^{\infty}
\exp\left(2(\lambda+i\eta)i\left|x-y\right|\right)\ 
\sign(x-y)
\  \sigma(y)\ dy
\\
&=
\frac{\exp(2\lambda i x)}{2} \int_{-\infty}^{x}
\exp\left(-2\eta \left|x-y\right|\right) \exp(-2\lambda iy)
\ \sigma(y)\ dy
\\
&-
\frac{\exp(-2\lambda i x)}{2} \int_{x}^{\infty}
\exp(-2 \eta \left|x-y\right|) 
\exp(2\lambda i y)
\  \sigma(y)\ dy
\end{aligned}
\label{bound:theorem4:2}
\end{equation}
 for all $x \in \mathbb{R}$. We observe that
\begin{equation}
\begin{aligned}
\int_{x}^{\infty}  \exp(-2 \eta \left|x-y\right|)  \exp(2\lambda i y) \  \sigma(y)\ dy
&=
\int_{-\infty}^\infty  \exp(-2 \eta \left|x-y\right|)  \exp(2\lambda i y) \  \sigma(y)\ dy
\\
&-
\int_{-\infty}^x  \exp(-2 \eta \left|x-y\right|)  \exp(2\lambda i y) \  \sigma(y)\ dy
\end{aligned}
\label{bound:theorem4:3}
\end{equation}
for all $x \in \mathbb{R}$.  
By inserting (\ref{bound:theorem4:3}) into (\ref{bound:theorem4:1}),
taking absolute values, and applying the triangle inequality
we obtain
\begin{equation}
\begin{aligned}
\left|\delta_\eta (x)\right|
&\leq
\frac{1}{2\lambda} 
\int_{-\infty}^{x}
\exp\left(-2\eta \left|x-y\right|\right)
\ \left|\sigma(y)\right|\ dy
\\
&+
\frac{1}{4\lambda}
\left|
\int_{-\infty}^\infty  \exp(-2 \eta \left|x-y\right|)  \exp(2\lambda i y) \  \sigma(y)\ dy
\right|
\end{aligned}
\label{bound:theorem4:4}
\end{equation}
for all $x \in \mathbb{R}$.
By inserting (\ref{bound:theorem4:3}) into (\ref{bound:theorem4:2}) and
taking absolute values, we conclude that
\begin{equation}
\begin{aligned}
\left|\delta_\eta' (x)\right|
&\leq
\int_{-\infty}^{x}
\exp\left(-2\eta \left|x-y\right|\right)
\ \left|\sigma(y)\right|\ dy
\\
&+
\frac{1}{2}
\left|
\int_{-\infty}^\infty  \exp(-2 \eta \left|x-y\right|)  \exp(2\lambda i y) \  \sigma(y)\ dy
\right|
\end{aligned}
\label{bound:theorem4:5}
\end{equation}
for all $x \in \mathbb{R}$.  We now combine  Theorem~\ref{bound:theorem3}
with our assumption that
\begin{equation}
\lambda \geq 
\frac{2}{\mu} \log\left(\frac{24\Gamma}{\eta}\right).
\end{equation}
which is part of (\ref{bound:theorem4:lambda}), in order to conclude that
\begin{equation}
\begin{aligned}
\left|
\int_{-\infty}^\infty  \exp(-2 \eta \left|x-y\right|)  \exp(2\lambda i y) \  \sigma(y)\ dy
\right|
\leq 
\frac{3\Gamma}{2} \exp\left(-\frac{\mu\lambda}{2}\right)
\leq \frac{\eta}{16}
\end{aligned}
\label{bound:theorem4:6}
\end{equation}
for all $x \in \mathbb{R}$.
We insert (\ref{bound:theorem4:6}) into (\ref{bound:theorem4:4}) and
(\ref{bound:theorem4:5}) in order to obtain the inequalities
\begin{equation}
\begin{aligned}
\left|\delta_\eta (x)\right|
&\leq
\frac{1}{2\lambda} 
\int_{-\infty}^{x}
\exp\left(-2\eta \left|x-y\right|\right)
\ \left|\sigma(y)\right|\ dy
+
\frac{\eta}{64\lambda}
\ \ \ \mbox{for all} \ \ x\in\mathbb{R}
\end{aligned}
\label{bound:theorem4:7}
\end{equation}
and
\begin{equation}
\begin{aligned}
\left|\delta_\eta' (x)\right|
&\leq \int_{-\infty}^{x} \exp\left(-2\eta \left|x-y\right|\right)
\ \left|\sigma(y)\right|\ dy +
\frac{\eta}{32}
\ \ \ \mbox{for all} \ \ x\in\mathbb{R}.
\end{aligned}
\label{bound:theorem4:8}
\end{equation}
Now we define $\delta$ via the formula
\begin{equation}
\delta(x) = T\left[\sigma\right](x).
\label{bound:theorem4:delta}
\end{equation}
From (\ref{main_theorem:inteq}) and (\ref{bound:theorem4:delta}) we conclude that 
\begin{equation}
\sigma(x) = S\left[\delta\right](x) + p(x) + \nu(x)
\ \ \ \mbox{for all}\ \ x\in\mathbb{R}.
\label{bound:theorem4:9}
\end{equation}
We combine conclusion (\ref{main_theorem:nu}) of Theorem~\ref{main_theorem}
with the assumption that
\begin{equation}
\lambda \geq
\frac{1}{\mu} \log\left(\frac{16\Gamma}{\mu \eta^2}\right)
\end{equation}
which is part of (\ref{bound:theorem4:lambda}), in order to conclude that
\begin{equation}
\left\|v\right\|_\infty \leq \frac{\Gamma}{\mu} \exp(-\mu\lambda)
\leq 
\frac{\eta^2}{16}.
\label{bound:theorem4:9.5}
\end{equation}
We combine (\ref{bound:theorem4:pandv}), (\ref{bound:theorem4:9.5}) and the fact that
\begin{equation}
\int_{-\infty}^x \exp(-2\eta \left|x-y\right|) dy = \frac{1}{2\eta}
\ \ \ \mbox{for all}\ \ x\in\mathbb{R}
\label{bound:theorem4:10}
\end{equation}
in order to conclude that
\begin{equation}
\int_{-\infty}^x \exp(-2\eta \left|x-y\right|) 
\left(\left|p(y)\right| + \left|\nu(y)\right|\right)\ dy
\leq 
\frac{\eta}{16}
\label{bound:theorem4:11}
\end{equation}
for all $x \leq x_0$.  
We combine Theorem~\ref{bound:theorem2} 
with (\ref{bound:theorem4:10}) in order to conclude that
\begin{equation}
\int_{-\infty}^x \exp\left(-2\eta\left|x-y\right|\right) \left|S\left[\delta\right](y)\right|\ dy
\leq 
\frac{\eta}{136}
+ \frac{102\lambda^2}{50\eta}
\sup_{-\infty \leq y \leq x} \left|\delta_\eta(y)\right|^2 
+ \frac{1}{4\eta}
\sup_{-\infty \leq y \leq x} \left|\delta_\eta'(y)\right|^2 
\label{bound:theorem4:12}
\end{equation}
for all $x \in \mathbb{R}$.  By combining (\ref{bound:theorem4:9}), 
(\ref{bound:theorem4:11}) and (\ref{bound:theorem4:12}) we 
conclude that
\begin{equation}
\int_{-\infty}^x \exp(-2\eta \left|x-y\right|)\left|\sigma(y)\right| \ dy
\leq 
\frac{19\eta}{272}
+
\frac{102 \lambda^2 }{50 \eta}
\sup_{-\infty \leq y \leq x}
\left|\delta_\eta(x)\right|^2 + 
\frac{1}{4\eta}
\sup_{-\infty \leq y \leq x}
\left|\delta_\eta'(x)\right|^2
\label{bound:theorem4:13}
\end{equation}
for all $x \leq x_0$.  Inserting (\ref{bound:theorem4:13}) into
(\ref{bound:theorem4:7}) and (\ref{bound:theorem4:8}) yields
the inequalities
\begin{equation}
\begin{aligned}
\left|\delta_\eta(x)\right|
\leq 
\frac{55\eta}{1088\lambda}
+
\frac{102 \lambda }{100 \eta }
\sup_{-\infty \leq y \leq x}
\left|\delta_\eta(x)\right|^2 + 
\frac{19}{272\eta \lambda}
\sup_{-\infty \leq y \leq x}
\left|\delta_\eta'(x)\right|^2
\end{aligned}
\label{bound:theorem4:14}
\end{equation}
and
\begin{equation}
\begin{aligned}
\left|\delta_\eta'(x)\right|
\leq 
\frac{9\eta}{136}
+
\frac{102 \lambda^2 }{50 \eta}
\sup_{-\infty \leq y \leq x}
\left|\delta_\eta(x)\right|^2 + 
\frac{1}{4\eta}
\sup_{-\infty \leq y \leq x}
\left|\delta_\eta'(x)\right|^2,
\end{aligned}
\label{bound:theorem4:15}
\end{equation}
both of which hold for all $x \leq x_0$.

We denote by $\Omega$ the set 
\begin{equation}
\left\{
x \leq x_0 : 
\left|\delta_\eta(y)\right| \leq \frac{\eta}{4 \lambda} \ \ \mbox{and}\ \ \ 
\left|\delta_\eta'(y)\right| \leq \frac{\eta}{4}\ \ \ \mbox{for all}\ \  
y \in (-\infty,x]
\right\}.
\label{bound:theorem4:16}
\end{equation}
Conclusion (\ref{bound:theorem1:conclusion1})
of Theorem~\ref{bound:theorem1} implies that $\Omega$ is nonempty.
We let $x^*$ denote the supremum of the set $\Omega$.  
Suppose  that $x^* < x_0$.  By inserting the inequalities
\begin{equation}
\left|\delta_\eta(y)\right|  \leq 
\frac{\eta}{4\lambda}
\ \ \ \mbox{for  all}\ \ y \leq x^*
\label{bound:theorem4:17}
\end{equation}
and
\begin{equation}
\left|\delta'_\eta(y)\right| \leq \frac{\eta }{4}
\ \ \ \mbox{for  all}\ \ y \leq x^*
\label{bound:theorem4:18}
\end{equation}
into (\ref{bound:theorem4:14}) and (\ref{bound:theorem4:15}) we conclude that
\begin{equation}
\begin{aligned}
\left|\delta_\eta(x)\right|
&\leq 
\frac{55\eta}{1088\lambda}
+
\frac{102  \eta}{1600 \lambda } 
+
\frac{\eta}{128 \lambda}
= \frac{6643\eta}{54400\lambda}
< \frac{\eta}{4\lambda}
\end{aligned}
\label{bound:theorem4:19}
\end{equation}
and
\begin{equation}
\begin{aligned}
\left|\delta_\eta'(x)\right|
\leq 
\frac{19 \eta}{272}
+
\frac{102 \eta}{800}
+
\frac{\eta}{64} 
= \frac{5693 \eta}{27200}
<
\frac{\eta}{4}
\end{aligned}
\label{bound:theorem4:20}
\end{equation}
for all $x \leq x^*$. 
We conclude from  (\ref{bound:theorem4:19}), (\ref{bound:theorem4:20})
and the continuity of  $\delta_\eta$ and $\delta_\eta'$ 
that there exists
a point $x_2 > x^*$ which is contained in $\Omega$.  This contradicts
our assumption that $x^* = \sup \Omega < x_0$, so
the supremum of $\Omega$ must, in fact, be $x_0$.
The conclusions 
 (\ref{bound:theorem4:conclusion1}) and (\ref{bound:theorem4:conclusion2}) follow
immediately.
\end{proof}


We now combine Theorems~\ref{bound:theorem1}, \ref{bound:theorem2},
\ref{bound:theorem3} and \ref{bound:theorem4} in order to establish
establish the principal result of this section, which is a bound
on the restriction of the nonoscillatory solution $\delta$ of the nonlinear differential
equation
\begin{equation}
\delta''(x) + 4\lambda^2 \delta(x) = S\left[\delta\right](x) + p(x) + \nu(x)
\end{equation}
to the interval $(-\infty,x_0]$ under the assumption that $p$ is of small magnitude there.

%
%

\vskip 1em
\begin{theorem}
Suppose that the hypotheses of Theorem~\ref{main_theorem} are satisfied, and that
$\sigma$ and $\nu$ are the functions obtained by invoking it.
Suppose further that $C_1$ is the real number
\begin{equation}
C_1 = \max\left\{24 \Gamma,\sqrt{\frac{16\Gamma}{\mu}}\right\},
\label{bound:theorem5:C1}
\end{equation}
and that 
\begin{equation}
\lambda \geq 
\max\left\{
4\Gamma,\frac{4}{\mu},
\frac{\Gamma}{4\mu},
\frac{2}{\mu} W_0\left(C_1\mu\right)
\right\}.
\label{bound:theorem5:lambda1}
\end{equation}
Suppose also that $x_0$ is a real number, that
\begin{equation}
\left| p(x)\right| \leq  
\frac{C_1^2}{16} \exp\left(-\mu\lambda\right)
\ \ \ \mbox{for all}\ \ x \leq x_0,
\label{bound:theorem5:pandv}
\end{equation}
and that $\delta$ is defined via the formula
\begin{equation}
\delta(x) = T\left[\sigma\right](x).
\label{bound:theorem5:delta_eta}
\end{equation}
Then 
\begin{equation}
\left|\delta(x)\right| \leq
\frac{C_1}{2\lambda} \exp\left(-\frac{\mu\lambda}{2}\right)
\label{bound:theorem5:conclusion1}
\end{equation}
and
\begin{equation}
\left|\delta'(x)\right| \leq
\frac{C_1}{2} \exp\left(-\frac{\mu\lambda}{2}\right)
\label{bound:theorem5:conclusion2}
\end{equation}
for all $x \leq x_0$.
\label{bound:theorem5}
\end{theorem}

\vskip 1em
\begin{remark}
In (\ref{bound:theorem5:lambda1}), $W_0$ refers to the branch of the Lambert $W$
function which is greater than or equal to $-1$ on the interval $[-1/e,\infty)$;
see Section~\ref{section:preliminaries:LambertW}.
\end{remark}

\begin{proof}
We let 
\begin{equation}
\eta = C_1 \exp\left(-\frac{\mu\lambda}{2}\right).
\label{bound:theorem5:1}
\end{equation}
From our assumption that
\begin{equation}
\lambda \geq \frac{2}{\mu}W_0\left(C_1 \mu \right),
\end{equation}
which is part of  (\ref{bound:theorem5:lambda1}),
 and Theorem~\ref{preliminaries:LambertW:theorem1}
we conclude that
\begin{equation}
\lambda \geq  2 C_1 \exp\left(\frac{-\mu\lambda}{2}\right) = 2\eta.
\label{bound:theorem5:1.5}
\end{equation}
Moreover, by inserting (\ref{bound:theorem5:C1}) into 
(\ref{bound:theorem5:1.5}) and invoking  Theorem~\ref{preliminaries:LambertW:theorem1}
we obtain
\begin{equation}
\eta \leq 24 \Gamma \exp\left(\frac{-\mu\lambda}{2}\right).
\label{bound:theorem5:1.6}
\end{equation}
and
\begin{equation}
\eta^2 \leq \frac{16\Gamma}{\eta} \exp\left(-\mu\lambda\right).
\label{bound:theorem5:1.7}
\end{equation}
It follows immediately from (\ref{bound:theorem5:1.6}) 
and (\ref{bound:theorem5:1.7}) that
\begin{equation}
\lambda \geq 
\max\left\{
\frac{2}{\mu} \log\left(\frac{24\Gamma}{\eta}\right),
\frac{1}{\mu} \log\left(\frac{16\Gamma}{\mu \eta^2}\right)
\right\}.
\label{bound:theorem5:2}
\end{equation}
Together (\ref{bound:theorem5:1}), (\ref{bound:theorem5:2})
and (\ref{bound:theorem5:lambda1}) ensure that 
the hypothesis (\ref{bound:theorem4:lambda}) of Theorem~\ref{bound:theorem4}
is satisfied.
From  (\ref{bound:theorem5:1}) and
(\ref{bound:theorem5:pandv}), we conclude that
\begin{equation}
\left|p(x)\right| \leq \frac{\eta^2}{16},
\label{bound:theorem5:4}
\end{equation}
so the hypothesis (\ref{bound:theorem4:pandv}) of Theorem~\ref{bound:theorem4}
is satisfied as well.
By invoking Theorem~\ref{bound:theorem4} we see that
\begin{equation}
\left|\delta_\eta(x) \right| \leq \frac{\eta}{4 \lambda}
\label{bound:theorem5:5}
\end{equation}
and
\begin{equation}
\left|\delta'_\eta(x) \right| \leq   \frac{\eta}{4}
\label{bound:theorem5:6}
\end{equation}
for all $x \leq x_0$.
We insert (\ref{bound:theorem5:1}) into (\ref{bound:theorem5:5}) and (\ref{bound:theorem5:6}) 
in order to see that
\begin{equation}
\left|\delta_\eta(x)\right| \leq
\frac{C_1}{4\lambda} \exp\left(-\frac{\mu\lambda}{2}\right)
\label{bound:theorem5:7}
\end{equation}
and
\begin{equation}
\left|\delta_\eta'(x)\right| \leq
\frac{C_1}{4} \exp\left(-\frac{\mu\lambda}{2}\right)
\label{bound:theorem5:8}
\end{equation}
for all $x \leq x_0$.  We combine 
the hypotheses (\ref{main_theorem:lambda}) of Theorem~\ref{main_theorem},
 and conclusions (\ref{bound:theorem1:conclusion4}) 
and (\ref{bound:theorem1:conclusion5}) 
of Theorem~\ref{bound:theorem1}   in order to obtain
\begin{equation}
\left|\delta(x)- \delta_\eta(x)\right| \leq \frac{3\Gamma\eta}{\pi \mu\lambda^2}
\leq  
  \frac{3}{16\pi} C_1\exp\left(-\frac{\mu\lambda}{2}\right).
\label{bound:theorem5:9}
\end{equation}
Similarly, from  (\ref{bound:theorem5:lambda1})
 and conclusions (\ref{bound:theorem1:conclusion4}) 
and (\ref{bound:theorem1:conclusion5}) 
of Theorem~\ref{bound:theorem1}   in order to obtain
we obtain
\begin{equation}
\left|\delta'(x)- \delta_\eta'(x)\right| \leq \frac{3 \Gamma \eta}{\mu^2 \pi \lambda^2}
\leq 
\frac{3}{16\pi}
 C_1\exp\left(-\frac{\mu\lambda}{2}\right).
\label{bound:theorem5:10}
\end{equation}
We combine (\ref{bound:theorem5:9}) with (\ref{bound:theorem5:7}) in order
to obtain (\ref{bound:theorem5:conclusion1}), and (\ref{bound:theorem5:10})
with (\ref{bound:theorem5:8}) in order to obtain (\ref{bound:theorem5:conclusion2}).
\end{proof}


\label{section:bound}
\end{subsection}

\begin{subsection}{A continuity result}

In this section, we use standard techniques from the theory of
ordinary differential equations to
bound the difference between the solution $\delta_0$ of 
the differential equation
\begin{equation}
\delta_0''(x) + 4\lambda^2 \delta_0(x) 
=
S\left[\delta_0\right](x) + p(x) + \eta(x)
\ \ \ \mbox{for all}\ \ x_0 \leq x \leq x_1
\label{continuity:diffeq}
\end{equation}
and the nonoscillatory solution $\delta$ of (\ref{kummer:deltaeq})
obtained from Theorem~\ref{main_theorem}
 under the assumptions that  the function $\eta$ and the quantities
\begin{equation}
\left|\delta(x_0) - \delta_0(x_0)\right|,\ 
\left|\delta'(x_0) - \delta_0'(x_0)\right|
\end{equation}
are of small magnitude.  
   We will also make the assumption
that the interval $[x_0,x_1]$ is of length less than or equal to $1$,
which is sufficient for our purposes.  Indeed, in all cases we will consider
the interval $[x_0,x_1]$ is contained in $[0,x(b)]$, where
$x$ is the function defined via the formula.
\begin{equation}
x(t) = \int_a^t \sqrt{q(u)}\ du.
\end{equation}
By scaling the parameter $\lambda^2$ and the coefficient $q$, we can 
assume without loss of generality that 
\begin{equation}
x(b) = \int_a^b \sqrt{q(u)}\ du \leq 1.
\end{equation}

%
%
\vskip 1em
\begin{theorem}
Suppose that $S$ is the nonlinear differential operator defined via (\ref{kummer:S}), and that
$f$, $g$ are continuously differentiable functions.  Then
\begin{equation}
\begin{aligned}
\left|S\left[f\right](x) \right.&\left.- S\left[g\right](x)\right|
\leq 
\frac{\left|f'(x) + g'(x)\right|}{4}
\left|f'(x)-g'(x)\right|\\
&+
2\lambda^2 \exp\left(\left|f(x)\right|\right)
 \exp\left(\left|f(x)-g(x)\right|\right)
\left|f(x) - g(x)\right|^2
\\
&+
4\lambda^2 \exp\left(\left|f(x)\right|\right)\left|f(x) - g(x)\right|
\left|f(x)\right|
\end{aligned}
\label{continuity:theorem1:conclusion}
\end{equation}
for all $x \in \mathbb{R}$.
\label{continuity:theorem1}
\end{theorem}
\begin{proof}
We define the operator $S_1$ via the formula
\begin{equation}
S_1\left[h\right](x) = \frac{(h'(x))^2}{4},
\label{continuity:theorem1:1}
\end{equation}
and the operator $S_2$ via the formula
\begin{equation}
S_2\left[h\right](x) = \exp(h(x))-h(x)-1
\label{continuity:theorem1:2}
\end{equation}
so that 
\begin{equation}
S\left[h\right](x) = S_1\left[h\right](x) + 4\lambda^2 S_2\left[h\right](x)
\label{continuity:theorem1:3}
\end{equation}
for all $x \in \mathbb{R}$.    We observe that
\begin{equation}
\left|S_1\left[f\right](x) - S_1\left[g\right](x)\right| \leq
\frac{\left|f'(x) + g'(x)\right|}{4}
\left|f'(x)-g'(x)\right|,
\label{continuity:theorem1:3.6}
\end{equation}
and that
\begin{equation}
\begin{aligned}
S_2\left[f\right](x) -  S_2\left[g\right](x)
&= \exp(f(x)) - \exp(g(x)) - (f(x) - g(x)) \\
&=
- \exp(f(x))\left(\exp(g(x)-f(x))\right)-(g(x)-f(x))-1)   
\\
& +\left(\exp(f(x))-1\right) \left(f(x) - g(x)\right).
\end{aligned}
\label{continuity:theorem1:4}
\end{equation}
By taking absolute values in (\ref{continuity:theorem1:4}) and inserting
the inequalities
\begin{equation}
\left|\exp(x)-x-1\right| \leq \frac{1}{2}\left|x\right|^2 \exp(\left|x\right|)
\end{equation}
and
\begin{equation}
\left|\exp(x)-1\right| \leq \left|x\right| \exp(\left|x\right|),
\end{equation}
we see that
\begin{equation}
\begin{aligned}
\left|S_2\left[f\right](x) -  S_2\left[g\right](x)\right|
&\leq
\exp\left(\left|f(x)\right|\right)\left|f(x) - g(x)\right|
\left(
\frac{\left|f(x) - g(x)\right|}{2} \exp\left(\left|f(x)-g(x)\right|\right)
+ \left|f(x)\right|
\right)
\end{aligned}
\label{continuity:theorem1:5}
\end{equation}
for all $x\in\mathbb{R}$.
We combine (\ref{continuity:theorem1:3.6}) and (\ref{continuity:theorem1:5})
in order to obtain (\ref{continuity:theorem1:conclusion}).
\end{proof}

%
%

\vskip 1em
\begin{theorem}
Suppose that the hypotheses of Theorems~\ref{main_theorem}
are satisfied, that $\sigma$ and $\nu$ are the functions obtained by invoking it,
and that $\delta$ is the function defined via the formula
\begin{equation}
\delta(x) = T\left[\sigma\right](x).
\label{continuity:theorem2:delta}
\end{equation}
Suppose also that $x_0 < x_1$ are real numbers such that
\begin{equation}
x_1-x_0 \leq 1,
\end{equation}
that  $0 < \epsilon < 1$  is a real number, that
\begin{equation}
\lambda \geq \max\left\{\epsilon,
\frac{6\Gamma}{\mu},
\frac{1}{u}\log\left(\frac{32\Gamma}{\mu\epsilon}\right)\right\},
\label{continuity:theorem2:lambda}
\end{equation}
and  that $\eta:[x_0,x_1] \to \mathbb{C}$ is a continuous function 
such that
\begin{equation}
\left|\eta(x)\right| \leq \frac{\epsilon}{32}
\ \ \ \mbox{for all}\ \ x_0 \leq x \leq x_1.
\label{continuity:theorem2:eta}
\end{equation}
 Suppose further that  $\alpha$ and $\beta$ are real numbers such that
\begin{equation}
\left|\delta(x_0)-\alpha\right|
\leq  \frac{\epsilon}{64\lambda}
\label{continuity:theorem2:alpha}
\end{equation}
and 
\begin{equation}
\left|\delta'(x_0)-\beta\right|
\leq \frac{\epsilon}{64}.
\label{continuity:theorem2:beta}
\end{equation}
Then there exists a twice continuously differentiable function
$\delta_0:[x_0,x_1]\to \mathbb{R}$ which solves
the initial value problem
\begin{equation}
\left\{
\begin{aligned}
\delta_0''(x) + 4\lambda^2 \delta_0(x) &= S\left[\delta_0\right](x) + p(x) + \eta(x)
\ \ \ \mbox{for all}\ \ x_0 \leq x \leq x_1
\\
\delta_0(x_0) &= \alpha \\
\delta_0'(x_0) &= \beta \\
\end{aligned}
\right.
\label{continuity:theorem2:ivp}
\end{equation}
and such that
\begin{equation}
\left|\delta(x)-\delta_0(x)\right| 
\leq \frac{\epsilon}{4\lambda}
\end{equation}
and
\begin{equation}
\left|\delta'(x)-\delta'_0(x)\right| 
\leq \frac{\epsilon}{2}
\label{continuity:theorem2:conclusion2}
\end{equation}
for all $x_0 \leq x \leq x_1$.
\label{continuity:theorem2}
\end{theorem}
\begin{proof}
From (\ref{continuity:theorem2:delta}) and Theorem~\ref{main_theorem}
we conclude that
\begin{equation}
\delta''(x) + 4\lambda^2 \delta(x) = S\left[\delta\right](x) + p(x) + \nu(x),
\label{continuity:theorem2:00}
\end{equation}
where
\begin{equation}
\left\|\nu\right\|_\infty \leq \frac{\Gamma}{\mu} \exp\left(-\mu \lambda\right).
\label{continuity:theorem2:01}
\end{equation}
We combine (\ref{continuity:theorem2:01}) with our assumption that
\begin{equation}
\lambda \geq \frac{1}{\mu} \log\left(\frac{32\Gamma}{\epsilon \mu}\right),
\end{equation}
which is part of (\ref{continuity:theorem2:lambda}), 
in order to conclude that 
\begin{equation}
\left\|\nu\right\|_\infty \leq \frac{\epsilon}{32}.
\label{continuity:theorem2:10.5}
\end{equation}

We let $\Omega$ be the set of all $y \in [x_0, x_1]$
such that there exists a twice continuously differentiable
function $\delta_0:[x_0,y]\to\mathbb{R}$ 
with the following properties:
\begin{equation}
\delta_0''(x) + 4\lambda^2 \delta_0(x) = S\left[\delta_0\right](x)
+ p(x) + \nu(x)
\ \ \ \mbox{for all}\ \ x_0 \leq x \leq y,
\label{continuity:theorem2:1}
\end{equation}
\begin{equation}
 \left|\delta(y)-\delta_0(y)\right| \leq 
\frac{\epsilon}{4\lambda}
\ \  \ \mbox{for all}\ \ x_0 \leq y \leq x,
\label{continuity:theorem2:2}
\end{equation}
and
\begin{equation}
 \left|\delta'(y)-\delta_0'(y)\right| \leq 
\frac{\epsilon}{2}
\ \  \ \mbox{for all}\ \ x_0 \leq y \leq x.
\label{continuity:theorem2:3}
\end{equation}
From  (\ref{continuity:theorem2:alpha}), (\ref{continuity:theorem2:beta})
and
the Picard-Lindel\"of Theorem (Theorem~\ref{preliminaries:picard:theorem1}
in Section~\ref{section:preliminaries:picard}),
we conclude that $\Omega$ is nonempty.
We let  $x^*$ be the supremum of the set $\Omega$.  
If $x^*=x_1$,  then conclusions of the theorem hold. 
We will suppose that $x^* < x_1$ and derive a contradiction.

We define the function $\tau$  via the formula
\begin{equation}
\tau(x) = \delta_0(x) - \delta(x),
\label{continuity:theorem2:4}
\end{equation}
and $\gamma$ via the formula
\begin{equation}
\gamma(x) = \tau(x_0) \cos(2\lambda (x-x_0)) + 
\frac{\tau'(x_0)}{2\lambda} \sin(2\lambda(x-x_0))
\label{continuity:theorem2:4.5}
\end{equation}
so that 
\begin{equation}
\gamma''(x) + 4\lambda^2 \gamma(x) = 0,
\label{continuity:theorem2:4.55}
\end{equation}
\begin{equation}
\tau(x_0) - \gamma(x_0) =  0 
\label{continuity:theorem2:4.6}
\end{equation}
and
\begin{equation}
\tau'(x_0) - \gamma'(x_0) =  0 .
\label{continuity:theorem2:4.7}
\end{equation}
From   (\ref{continuity:theorem2:alpha}), (\ref{continuity:theorem2:beta})
and (\ref{continuity:theorem2:4})
we see that
\begin{equation}
\left| \tau(x_0) \right| \leq 
\frac{\epsilon}{64\lambda}
\label{continuity:theorem2:5}
\end{equation}
and
\begin{equation}
\left| \tau'(x_0) \right| \leq 
\frac{\epsilon}{64}.
\label{continuity:theorem2:6}
\end{equation}
By combining (\ref{continuity:theorem2:5}), (\ref{continuity:theorem2:6})
and (\ref{continuity:theorem2:4.5}), we see that
\begin{equation}
\left|\gamma(x)\right| \leq \frac{3\epsilon}{128\lambda}
\ \ \ \mbox{for all}\ \ x\in\mathbb{R},
\label{continuity:theorem2:6.5}
\end{equation}
and that 
\begin{equation}
\left|\gamma'(x)\right| \leq \frac{3\epsilon}{64}
\ \ \ \mbox{for all}\ \ x\in\mathbb{R}.
\label{continuity:theorem2:6.6}
\end{equation}
We combine 
(\ref{continuity:theorem2:00}),
(\ref{continuity:theorem2:1}) and (\ref{continuity:theorem2:4.55})
in order to conclude that 
\begin{equation}
\tau''(x)-\gamma''(x)  + 4\lambda^2 (\tau(x) - \gamma(x)) = 
S\left[\delta_0\right](x) -  S\left[\delta\right](x) + \eta(x) - \nu(x)
\label{continuity:theorem2:8}
\end{equation}
for all $x_0 \leq x \leq x^*$.  
By invoking Theorem~\ref{preliminaries:helmholtz:theorem1} of
Section~\ref{section:preliminaries:helmholtz}, we see that 
\begin{equation}
\begin{aligned}
\tau(x) - \gamma(x) &= \tau(x_0) - \gamma(x_0) + (\tau'(x_0)- \gamma'(x_0)) (x-x_0) \\
&+
\frac{1}{2\lambda} \int_{x_0}^x \sin\left(2\lambda(x-y)\right)\ 
\left(S\left[\delta_0\right](y) - S\left[\delta\right](y) + \eta(y) - \nu(y)\right)\ dy
\end{aligned}
\label{continuity:theorem2:8.5}
\end{equation}
for all $x_0 \leq x \leq x^*$.  We combine (\ref{continuity:theorem2:8.5})
with (\ref{continuity:theorem2:4.6}) and (\ref{continuity:theorem2:4.7})
in order to conclude that
\begin{equation}
\begin{aligned}
\tau(x) &= \gamma(x) + 
\frac{1}{2\lambda} \int_{x_0}^x \sin\left(2\lambda(x-y)\right)\ 
\left(S\left[\delta_0\right](y) - S\left[\delta\right](y) + \eta(y) - \nu(y)\right)\ dy
\end{aligned}
\label{continuity:theorem2:9}
\end{equation}
for all $x_0 \leq x \leq x^*$
We differentiate (\ref{continuity:theorem2:9})
in order to conclude that
\begin{equation}
\tau'(x) = \gamma'(x)+
\int_{x_0}^x \cos\left(2\lambda(x-y)\right)\ 
\left(S\left[\delta_0\right](y) - S\left[\delta\right](y) + \eta(y) - \nu(y)\right)\ dy
\label{continuity:theorem2:10}
\end{equation}
for all $x_0 \leq x \leq x^*$.  

Next, we observe that Theorem~\ref{continuity:theorem1} implies
that

\begin{equation}
\begin{aligned}
\left|S\left[\delta_0\right](x)  - S\left[\delta\right](x)\right|
&\leq
\frac{\left|\delta'(x)\right| |\tau'(x)|}{2}
+
\frac{\left|\tau'(x)\right|^2}{4}
+ 2\lambda^2 
\exp(|\delta(x)|)\exp(\left|\tau(x)\right|) |\tau(x)|^2
\\
&
+ 4\lambda^2 
\exp(\left|\delta(x)\right|) |\tau(x)| |\delta(x)|
\end{aligned}
\label{continuity:theorem2:11}
\end{equation}

for all $x \in \mathbb{R}$.
We combine the hypotheses (\ref{main_theorem:lambda}) 
of Theorem~\ref{main_theorem} with the conclusions (\ref{bound:theorem1:conclusion2})  and
(\ref{bound:theorem1:conclusion3})
of Theorem~\ref{bound:theorem1} 
and (\ref{continuity:theorem2:lambda})
in order to obtain
\begin{equation}
\left\|\delta\right\|_\infty \leq \frac{3\Gamma}{4\pi \mu\lambda^2}
\leq \frac{3}{64\pi},
\label{continuity:theorem2:12}
\end{equation}
\begin{equation}
\left\|\delta\right\|_\infty \leq \frac{3\Gamma}{4\pi \mu\lambda^2}
\leq \frac{1}{8\pi\lambda}
\label{continuity:theorem2:13}
\end{equation}
and
\begin{equation}
\left\|\delta'\right\|_\infty 
\leq 
\frac{3\Gamma }{4\pi \mu^2 \lambda^2}
\leq 
\frac{1}{32}.
\label{continuity:theorem2:14}
\end{equation}
By inserting (\ref{continuity:theorem2:12})
(\ref{continuity:theorem2:13}) and (\ref{continuity:theorem2:14}) into 
(\ref{continuity:theorem2:11}) we obtain the inequality
\begin{equation}
\begin{aligned}
\left|S\left[\delta_0\right](x)  - S\left[\delta\right](x)\right|
&\leq
\frac{1}{64} |\tau'(x)|
+
\frac{\left|\tau'(x)\right|^2}{4}
+ 2\lambda^2 
\exp\left(\frac{3}{64\pi}\right)\exp(\left|\tau(x)\right|) |\tau(x)|^2
\\
&
+
\frac{\lambda}{2}
\exp\left(\frac{3}{64\pi}\right)|\tau(x)|
\\
&\leq
\frac{1}{64} |\tau'(x)|
+
\frac{\left|\tau'(x)\right|^2}{4}+
\frac{51\lambda^2}{25}\exp(\left|\tau(x)\right|) |\tau(x)|^2
+
\frac{51\lambda}{100\pi}|\tau(x)|,
\end{aligned}
\label{continuity:theorem2:15}
\end{equation}
which holds for all $x \in \mathbb{R}$.  In the second inequality
of (\ref{continuity:theorem2:15}) we used (\ref{continuity:theorem2:12})
and the fact that 
\begin{equation}
\exp\left(\frac{3}{64\pi}\right)< \frac{51}{50}.
\label{continuity:theorem2:16}
\end{equation}

We insert
(\ref{continuity:theorem2:eta}),
(\ref{continuity:theorem2:10.5}),
(\ref{continuity:theorem2:6.5})
and
(\ref{continuity:theorem2:15})
into (\ref{continuity:theorem2:9})
in order to conclude that
\begin{equation}
\begin{aligned}
\left|\tau(x)\right|
\leq 
\frac{\epsilon}{16\lambda}
&+
\frac{1}{128\lambda} \sup_{x_0 \leq y \leq x} |\tau'(y)|
+
\frac{1}{8\lambda} \sup_{x_0 \leq y \leq x} \left|\tau'(y)\right|^2\\
&+
\frac{51\lambda}{50} \sup_{x_0 \leq y \leq x} \exp(\left|\tau(y)\right|) |\tau(y)|^2
+\frac{51}{200\pi}
 \sup_{x_0 \leq y \leq x} \left|\tau(y)\right|
\end{aligned}
\label{continuity:theorem2:17}
\end{equation}
for all $x_0 \leq x \leq x^*$.  By inserting (\ref{continuity:theorem2:2}) 
and  (\ref{continuity:theorem2:3})  into (\ref{continuity:theorem2:17}),
we see that
\begin{equation}
\begin{aligned}
\left|\tau(x)\right|
\leq 
\frac{\epsilon}{16\lambda}
&+
\frac{\epsilon}{256\lambda}
+
\frac{\epsilon^2}{32\lambda} 
+
\frac{51\epsilon^2}{800\lambda}
\exp\left(\frac{1}{4}\right)
+\frac{51\epsilon}{800\pi\lambda}
<
\frac{\epsilon}{4\lambda}
\end{aligned}
\label{continuity:theorem2:18}
\end{equation}
for all $x_0 \leq x \leq x^*$.
Similarly, we combine
(\ref{continuity:theorem2:eta}),
(\ref{continuity:theorem2:10.5}),
(\ref{continuity:theorem2:6.6}),
(\ref{continuity:theorem2:15})
and  (\ref{continuity:theorem2:10})
in order to conclude that 
\begin{equation}
\begin{aligned}
\left|\tau'(x)\right|
\leq 
\frac{\epsilon}{8}
&+
\frac{1}{64} \sup_{x_0 \leq y \leq x_0+\Delta} |\tau'(y)|
+
\frac{1}{4} \sup_{x_0 \leq y \leq x} \left|\tau'(y)\right|^2
\\
&+
\frac{51\lambda^2}{25} \sup_{x_0 \leq y \leq x} \exp(\left|\tau(y)\right|) |\tau(y)|^2
+\frac{51\lambda}{100\pi}
 \sup_{x_0 \leq y \leq x} \left|\tau(y)\right|
\end{aligned}
\label{continuity:theorem2:19}
\end{equation}
for all $x_0 \leq x \leq x^*$.  
By inserting (\ref{continuity:theorem2:2}) 
and  (\ref{continuity:theorem2:3})  into (\ref{continuity:theorem2:19})
we see that 
\begin{equation}
\left|\tau'(x)\right|
\leq 
\frac{\epsilon}{8}
+
\frac{\epsilon}{128}
+
\frac{\epsilon^2}{16}
+
\frac{51\epsilon^2}{400}
\exp\left(\frac{1}{4}\right)
+
\frac{51\epsilon}{400\pi}
< \frac{\epsilon}{2}
\label{continuity:theorem2:20}
\end{equation}
for all $x_0 \leq x \leq x^*$.
The Picard-Lindel\"of theorem together with (\ref{continuity:theorem2:17})
and (\ref{continuity:theorem2:20}) imply that there exists a $y > x^*$ such 
that (\ref{continuity:theorem2:2}) and (\ref{continuity:theorem2:3}) hold.
This is a contradiction since $x^* = \sup \Omega$.
\end{proof}

The following result is obtained by letting
\begin{equation}
\epsilon = 24 \Gamma \exp\left(-\lambda \mu\right)
\end{equation}
in Theorem~\ref{continuity:theorem2}.

%
%

\vskip 1em
\begin{theorem}
Suppose that the hypotheses of Theorems~\ref{main_theorem}
are satisfied, that $\sigma$ and $\nu$ are the functions obtained by invoking it,
and that $\delta$ is the function defined via the formula
\begin{equation}
\delta(x) = T\left[\sigma\right](x).
\label{continuity:theorem3:delta}
\end{equation}
Suppose further that $C_2$ is the real number
\begin{equation}
C_2 = 24\Gamma,
\label{continuity:theorem3:C2}
\end{equation}
 that $x_0 < x_1$ are real numbers such that
\begin{equation}
x_1-x_0 \leq 1,
\end{equation}
that 
\begin{equation}
\lambda \geq \max\left\{
\frac{6\Gamma}{\mu},
\frac{1}{\mu}W_0\left(\frac{C_2}{\mu}\right)
\right\}
\label{continuity:theorem3:lambda}
\end{equation}
and  that $\eta:[x_0,x_1] \to \mathbb{C}$ is a continuous function 
such that
\begin{equation}
\left|\eta(x)\right| \leq \frac{C_2}{32}\exp\left(-\lambda \mu\right)
\ \ \ \mbox{for all}\ \ x_0 \leq x \leq x_1.
\label{continuity:theorem3:eta}
\end{equation}
 Suppose further also  $\alpha$ and $\beta$ are real numbers such that
\begin{equation}
\left|\delta(x_0)-\alpha\right|
\leq  \frac{C_2}{64\lambda}\exp\left(-\lambda \mu\right)
\label{continuity:theorem3:alpha}
\end{equation}
and 
\begin{equation}
\left|\delta'(x_0)-\beta\right|
\leq \frac{C_2}{64}\exp\left(-\lambda \mu\right).
\label{continuity:theorem3:beta}
\end{equation}
Then there exists a twice continuously differentiable function
$\delta_0:[x_0,x_1]\to \mathbb{R}$ which solves
the initial value problem
\begin{equation}
\left\{
\begin{aligned}
\delta_0''(x) + 4\lambda^2 \delta_0(x) &= S\left[\delta_0\right](x) + p(x) + \eta(x)
\ \ \ \mbox{for all}\ \ x_0 \leq x \leq x_1
\\
\delta_0(x_0) &= \alpha \\
\delta_0'(x_0) &= \beta, \\
\end{aligned}
\right.
\label{continuity:theorem3:ivp}
\end{equation}
and such that
\begin{equation}
\left|\delta(x)-\delta_0(x)\right| 
\leq \frac{C_2}{4\lambda}\exp\left(-\lambda \mu\right)
\end{equation}
and
\begin{equation}
\left|\delta'(x)-\delta'_0(x)\right| 
\leq \frac{C_2}{2}\exp\left(-\lambda \mu\right).
\label{continuity:theorem3:conclusion2}
\end{equation}
for all $x_0 \leq x \leq x_1$.
\label{continuity:theorem3}
\end{theorem}

\vskip 1em
\begin{remark}
In (\ref{continuity:theorem3:lambda}), $W_0$ refers to the branch of the Lambert $W$
function which is greater than or equal to $-1$ on the interval $[-1/e,\infty)$;
see Section~\ref{section:preliminaries:LambertW}.
\end{remark}

\label{section:continuity}
\end{subsection}

\label{section:apparatus}

\end{section}


\begin{section}{Numerical algorithm}

In this section, we describe an algorithm for the solution of
the boundary value problem
\begin{equation}
\left\{
\begin{aligned}
y''(t) + \lambda^2 q(t) y(t) &= 0 \ \ \mbox{for all}\ \ a \leq t \leq b \\
c_1 y(a) + c_2 y'(a) &= \alpha \\
c_3 y(b) + c_4 y'(b) &= \beta.
\end{aligned}
\right.
\label{algorithm:bvp}
\end{equation}
where  $c_1$, $c_2$, $c_3$, $c_4$, 
$\alpha$, $\beta$ and $\lambda > 0$ are real numbers, and 
$q$ is a strictly positive on the interval $[a,b]$ 
and analytic in an open set containing the interval
$[a,b]$.
It can be easily modified to address, {\it inter alia}, initial value problems.

The algorithm exploits the analytical appparatus 
developed in Section~\ref{section:apparatus}
in order to construct a solution $r_2$ of the logarithm
form of Kummer's equation
\begin{equation}
r_2''(t) - \frac{1}{4} (r_2'(t))^2 + 4\lambda^2 \left(\exp(r_2(t)) - q(t)\right) = 0
\ \ \ \mbox{for all}\ \ \ a \leq t \leq b.
\label{algorithm:kummer_log}
\end{equation}
Once the function $r_2$ has been obtained, we construct a phase function $\alpha$ via the
formula
\begin{equation}
\alpha(t) = \gamma \int_0^t \exp\left(\frac{r_2(u)}{2}\right)\ du.
\label{algorithm:phase_formula}
\end{equation}
It has the property that the functions $u$, $v$ defined by the formulas
\begin{equation}
u(t) = \frac{\cos(\alpha(t)) }{\left|\alpha'(t)\right|^{1/2}}
\label{algorithm:u}
\end{equation}
and
\begin{equation}
v(t) = \frac{\sin(\alpha(t)) }{\left|\alpha'(t)\right|^{1/2}}
\label{algorithm:v}
\end{equation}
form a basis in the space of solutions of the ordinary 
differential equation
\begin{equation}
y''(t) + \lambda^2 q(t) y(t) = 0 \ \ \mbox{for all}\ \ a \leq t \leq b.
\end{equation}
Real numbers $d_1$ and $d_2$ such that the function
\begin{equation}
y_0(t) = d_1 u(t) + d_2 v(t)
\label{algorithm:solution}
\end{equation}
satisfies the boundary conditions
\begin{equation}
\begin{aligned}
c_1 y(a) + c_2 y'(a) &= \alpha \\
c_3 y(b) + c_4 y'(b) &= \beta
\end{aligned}
\label{algorithm:conditions}
\end{equation}
are calculated in the obvious fashion: by inserting
(\ref{algorithm:solution}) into (\ref{algorithm:conditions}),
evaluating the functions $u$ and $v$ at the points
$a$ and $b$ via formulas (\ref{algorithm:u}) and (\ref{algorithm:v}), 
and solving the resulting
system of two linear algebraic equations in the two unknowns $d_1$, $d_2$.

In addition to the value of $\lambda$ and a routine for evaluating the function $q$
at any point on the interval $[a,b]$, the  user supplies as inputs to the algorithm
an integer $m>0$ and a partition
\begin{equation}
a = \xi_0 < \xi_1 < \xi_2 < \cdots < \xi_n = b
\end{equation}
of the interval $[a,b]$.  For each $j=1,\ldots,n$, the restrictions of the functions $r$ and 
$\alpha$ to $\left[\xi_{j-1},\xi_{j}\right]$ 
are represented by their values at the points
\begin{equation}
x_{j,0},x_{j,1}, x_{j,2}, \ldots, x_{j,m}, x_{j,m}
\label{algorithm:chebyshev}
\end{equation}
of the $(m+1)$-point Chebyshev grid on the interval
$\left[\xi_{j-1},\xi_{j}\right]$ (see Section~\ref{section:preliminaries:chebyshev}).
The  assumption is, of course, that the restrictions of 
 these functions to each subinterval 
are  well-approximated by polynomials of degree $m$.
Note that for each $j=1,\ldots,n-1$,
the last Chebyshev point in the interval $\left[\xi_{j-1},\xi_{j}\right]$
coincides with the first Chebyshev point in the 
interval $\left[\xi_{j},\xi_{j+1}\right]$; that is,
\begin{equation}
x_{j,m} = \xi_{j+1} = x_{j+1,0}
\end{equation}
for all $j=1,\ldots,n-1$.  

The output of the algorithm consists of the values
of $\alpha$ and $\alpha'$ at each of the $n \cdot (m+1)$ points
\begin{equation}
x_{1,0},\ldots,x_{1,m}, 
x_{2,0},\ldots,x_{2,m},
\ldots
x_{n,0},\ldots,x_{n,m}.
\label{algorithm:discretization}
\end{equation}
Using this data, the value of the solution $y_0$ of the boundary value
problem (\ref{algorithm:bvp}) 
can be computed at any point
$t$ in $[a,b]$.  More specifically, to evaluate  $y_0(t)$
at the point $t$, 
we  calculate
 $\alpha(t)$ and $\alpha'(t)$ via Chebyshev interpolation  
(as discussed in Section~\ref{section:preliminaries:chebyshev}),
then evaluate $u(t)$ and $v(t)$ using formulas
(\ref{algorithm:u}) and (\ref{algorithm:v}),
and finally 
  insert the values of $u(t)$ and $v(t)$ into 
(\ref{algorithm:solution}) in order to obtain $y_0(t)$.

Our algorithm calls for solving a number of stiff ordinary differential equations.
In our implementation, we used
the spectral deferred correction method described in
\cite{Dutt-Greengard-Rokhlin}.
It was chosen for its excellent stability properties;
however, any standard approach to the numerical solution of 
stiff ordinary differential equation can be substituted for
the algorithm of \cite{Dutt-Greengard-Rokhlin}.

We now describe the procedure for the construction of the phase
function $\alpha$ in detail.  It consists of the following
four phases.

\noindent 
{\it Phase 1: Construction of the Windowed Problem}

In the first phase of the algorithm we construct a windowed
version $\tilde{q}$ of the function $q$ using the following
sequence of steps:

\begin{enumerate}
\item 
We let 
%
\begin{equation}
\psi(t) = 
\frac{1-\erf\left(\frac{13}{b-a} \left(t-\frac{a+b}{2}\right)\right)}{2}
\label{algorithm:erf}
\end{equation}
so that $\psi(t) \approx 1$ for all $t$ near $a$
and $\psi(t) \approx 0$ for all $t$ near $b$.  Note that
the constant $13$ in (\ref{algorithm:erf}) was chosen to be the 
smallest positive integer such that 
the quantities $\left|1-\phi(a) \right|$
and $\left|\phi(b) \right|$ are less than machine precision.

\item
We define the function $\tilde{q}$ by the  formula
\begin{equation}
\tilde{q}(t) = \psi(t) + (1-\psi(t)) q(t)
\label{algorithm:windowedq}
\end{equation}
so that  $\tilde{q}(t) \approx 1$ when $t$ is close to $a$ and
$\tilde{q}(t) \approx q(t)$ when $t$ is close to $b$.
We refer to $\tilde{q}$ as the windowed version of $q$.  
\end{enumerate}

\vskip 1em\noindent
{\it Phase 2: Solution of the windowed problem }

In this phase, we solve the initial value problem 
\begin{equation}
\left\{
\begin{aligned}
r_1''(t) - \frac{1}{4}\left(r_1'(t)\right)^2 + 
4 \lambda^2\left( \exp(r_1(t)) - \tilde{q}(t)\right) &= 0
\ \ \ \mbox{for all}\ \ a \leq t \leq b\\
r_1(a) = r_1'(a) &= 0,
\end{aligned}
\right.
\label{algorithm:kummer_log_form}
\end{equation}
with the windowed function $\tilde{q}$ is in place of the original function $q$.
We denote by $\tilde{r}$ the nonoscillatory 
solution of the logarithm form of Kummer's equation obtained by applying
Theorem~\ref{main_theorem} to the second order ordinary differential equation
\begin{equation}
y''(t) + \lambda^2 \tilde{q}(t) y(t) = 0.
\end{equation}
By invoking Theorems~\ref{bound:theorem5} and \ref{continuity:theorem3}, 
we see that 
\begin{equation}
\left|r_1(t) - \tilde{r}(t) \right| 
+
\left|r_1'(t) - \tilde{r}'(t) \right| 
= \mathcal{O} \left(\exp\left(-\frac{\lambda \mu}{2}\right)\right)
\label{algorithm:r1}
\end{equation}
for all $t$ close to $b$.
Assuming that $\lambda$ is sufficiently large, 
the difference between $r_1$ and the  nonoscillatory function $\tilde{r}$
is well below  machine precision and $r_1$ can be 
treated as nonoscillatory for the purposes of numerical computation.

For each $j=1,\ldots,n$, we compute the solution of 
 (\ref{algorithm:kummer_log_form}) at the points
\begin{equation}
x_{j,0}, \ldots, x_{j,m}
\end{equation}
of the $(m+1)$-point Chebyshev grid on  $\left[\xi_{j-1},\xi_{j}\right]$.
If $j=1,$ then the initial conditions are taken to be
\begin{equation}
\begin{aligned}
r_1(a)   = r_1'(a)= 0.
 \end{aligned}
\end{equation}
If, on the other hand, $j>1$, then we enforce the conditions
\begin{equation}
r_1\left(\xi_{j,1}\right) = r\left(\xi_{j-1,m}\right) 
\end{equation}
and
\begin{equation}
r_1'\left(\xi_{j,1}\right) = r_1'\left(\xi_{j-1,m}\right);
\end{equation}
that is,  we require that 
$r_1$ and its first derivative
 agree at the left endpoint of the  interval
 with the value and derivative of the solution
at the right endpoint of the previous interval.

\vskip 1em \noindent
{\it Phase 3: Solution of the original problem}

In this phase, we solve the  problem
\begin{equation}
\left\{
\begin{aligned}
r_2''(t) - \frac{1}{4}\left(r_2'(t)\right)^2 + 4 \lambda^2\left( \exp(r_2(t)) - q(t)\right) &= 0
\ \ \ \mbox{for all}\ \ a \leq t \leq b \\
r_2(b) = r_1(b)\\
r_2'(b) = r_1'(b)\\
\end{aligned}
\right.
\label{algorithm:kummer_log_form2}
\end{equation}
The intervals are processed in decreasing order:
 the $n^{th}$ interval
$\left[\xi_{n-1},\xi_n\right]$ is the first to be processed, then
$\left[\xi_{n-2},\xi_{n-1}\right]$, and so on.
Boundary conditions are imposed at the left end point of each interval;
in particular,
when processing the $n^{th}$ interval we require that 
\begin{equation}
\begin{aligned}
r_2\left(\xi_n\right)  &= r_1\left(\xi_n\right)\\
r_2'\left(\xi_n\right) &= r_1\left(\xi_n\right)
\end{aligned}
\end{equation}
and while processing each of the subsequent intervals
$\left[\xi_{j-1},\xi_{j}\right]$ we require that
\begin{equation}
\begin{aligned}
r_2\left(x_{j,m}\right) &= r_2\left(x_{j+1,0}\right)\\
r_2'\left(x_{j,m}\right) &= r_2'\left(x_{j+1,0}\right).\\
\end{aligned}
\end{equation}
We combine Theorems~\ref{bound:theorem5} and \ref{continuity:theorem3}
with (\ref{algorithm:r1})
in order to conclude that 
\begin{equation}
\left|r(t) - r2(t) \right| = \mathcal{O} \left( \exp\left(-\frac{\mu\lambda}{2}\right)\right)
\label{algorithm:r2}
\end{equation}
for all $a \leq t \leq b$.
As in the case of $r_1$, (\ref{algorithm:r2}) implies that in the high-frequency regime, the 
difference between 
$r_2$ and the nonoscillatory solution $r$ of the logarithm form of Kummer's
equation associated with the coefficient $q$  is much smaller than
machine precision.  Consequently, we regard $r_2$ as nonoscillatory
for the purposes of numerical computation.

\noindent 
{\it Phase 4: Preparation of the output}

In this final phase, the values of the functions $\alpha$ and $\alpha'$ 
are tabulated at each of the points (\ref{algorithm:discretization})
via the following sequence of steps:

\begin{enumerate}

\item
We compute the values of $\alpha'$ at the 
points (\ref{algorithm:discretization}) using the 
formula  
\begin{equation}
\alpha'(t)  = \lambda \exp\left(\frac{r_2'(t)}{2}\right).
\end{equation}

\item
For each $j=1,\ldots,n$, we apply the spectral integration
matrix of order $m$ (see Section~\ref{section:preliminaries:chebyshev})
to the vector
\begin{equation}
\left(
\begin{array}{c}
\alpha'\left(x_{j,0}\right)\\
\alpha'\left(x_{j,1}\right)\\
\vdots \\
\alpha'\left(x_{j,m}\right)
\end{array}
\right)
\end{equation}
in order to obtain the values 
\begin{equation}
\alpha_j\left(x_{j,0}\right), \alpha_j\left(x_{j,1}\right), \ldots,
\alpha_j\left(x_{j,m}\right)
\end{equation}
of an antiderivative $\alpha_j$ of the restriction of $\alpha'$  
to the interval $\left[\xi_{j-1},\xi_j\right]$
at the nodes of the $(m+1)$-point Chebyshev grid
on that interval.
Note that the value of  $\alpha_j(\xi_j)$ is not necessarily
consistent with the value of $\alpha_{j+1}(\xi_j)$.  This problem
is  corrected in the following steps.

\item
For each $j=1,\ldots,n$,  we define a real number $\gamma_j$ as follows
\begin{equation}
\begin{cases}
\gamma_j = \alpha \left(\xi_{1,0}\right) & \ \ \mbox{if}\ \ j = 1 \\
\gamma_j = \alpha \left(\xi_{j-1,m}\right) & \ \ \mbox{if}\ \  j >1 \\
\end{cases}
\end{equation}
\item
For each $j=2,\ldots,n$ and each $i=0,\ldots,m$, the value
of the phase function $\alpha$ at the point $x_{j_,i}$ is computed
via the formula
\begin{equation}
\alpha\left(x_{j,i}\right) = \alpha_j\left(x_{j,i}\right) - \alpha_j\left(x_{j,0}\right) + \gamma_j.
\end{equation}

\end{enumerate}

The output of the algorithm consists of the values of $\alpha'$
at the nodes (\ref{algorithm:chebyshev})
computed in Step 1 of Phase 4 and the values of $\alpha$ 
at the nodes (\ref{algorithm:chebyshev}) computed in Step 4 of 
Phase 4.  

\label{section:algorithm}
\end{section}

\begin{section}{Numerical experiments}

In this section, we describe numerical experiments 
performed to evaluate
the performance of the algorithm of Section~\ref{section:algorithm}.
Our code  was written in Fortran and
 compiled with the Intel Fortran Compiler version 13.1.3.
All calculations were carried out 
on a desktop computer equipped with an Intel Xeon X5690 CPU 
running at 3.47 GHz.    
Unless otherwise noted, double precision (Fortran REAL*8)
arithmetic was used.

%
%
\begin{subsection}{Comparison with a standard solver}
We measured the performance of the algorithm of this paper
by applying it to the initial value problem
\begin{equation}
\left\{
\begin{aligned}
y''(t) + \lambda^2 q(t) y(t) &= 0 
\ \ \ \mbox{for all}\ \ -1 \leq t \leq 1 \\
y(-1) &= 0\\
y'(-1) &= \lambda,
\end{aligned}
\right.
\label{experiments:simple:ode}
\end{equation}
where $q$ is defined by the formula
\begin{equation}
q(t) = 1 - t^2\cos(3t),
\label{experiments:simple:q}
\end{equation}
for seven values of $\lambda$.
A reference solution was obtained by executing
the spectral deferred correction method of \cite{Dutt-Greengard-Rokhlin}
in extended precision (Fortran REAL*16) arithmetic.
The interval $[-1,1]$ was partitioned into $10$ equispaced
subintervals and the $16$ point Chebyshev grid was used to represent
the nonoscillatory phase function on each subinterval.
For each value of $\lambda$, the obtained solution was compared
to the reference solution at $1000$ randomly chosen  points on the interval $[-1,1]$.

The results of this experiment are reported in 
Table~\ref{numerics:simple:table}.  Each row there corresponds to one value
of $\lambda$ and reports the time required to construct the nonoscillatory
phase function, the average time required to evaluate the solution of 
(\ref{experiments:simple:ode}) using this nonoscillatory phase function, 
and the maximum absolute error 
which was observed.
We see that the time required to solve (\ref{experiments:simple:ode})
was independent of the value of the parameter $\lambda$,
and that the obtained accuracy decreased as $\lambda$ increased.
This loss of precision was incurred when the sine and cosine of large arguments
were calculated in the course of 
evaluating the functions $u$, $v$ defined via
 formulas (\ref{introduction:u}), (\ref{introduction:v}).

Plots of the function $q$ defined by (\ref{experiments:simple:q})
and the windowed version of $q$ constructed as an intermediate
step by the algorithm of Section~\ref{section:algorithm}
are shown in Figure~\ref{figure:simple:q}.  Plots
of the solution $r$ of the logarithm form of Kummer's equation
when $\lambda=10^7$
and the windowed version $r_1$ of $r$ constructed as an intermediate
step by the algorithm of Section~\ref{section:algorithm}
are shown in Figure~\ref{figure:simple:r}.

\label{section:experiments:simple}
\end{subsection}

%
%

\begin{subsection}{Phase functions for Chebyshev's equation}

Chebyshev's equation
\begin{equation}
(1-t^2)y''(t) - t y'(t) + \lambda^2 y(t) = 0
\ \ \ \mbox{for all} \ \ -1 \leq  t \leq 1
\label{experiments:chebyshev:1}
\end{equation}
admits an exact nonoscillatory phase function which can be represented
via elementary functions.  More specifically,
\begin{equation}
\alpha_0(t) = \lambda \arccos(t)
\label{experiments:chebyshev:2}
\end{equation}
is a nonoscillatory phase function for the second order equation
\begin{equation}
\psi''(t) + 
\left(
\frac{2 + t^2 + 4 \lambda^2 \left(1-t^2\right)}{4\left(1-t^2\right)^2}
\right)
\psi(t) = 0
\ \ \ \mbox{for all}\ \ \ -1 \leq t \leq 1
\label{experiments:chebyshev:3}
\end{equation}
obtained by introducing
\begin{equation}
\psi(t) = (1-t^2)^{1/4} y(t)
\end{equation}
into  (\ref{experiments:chebyshev:1}).
For each $\lambda = 10, 20, \ldots, 1000$,
we applied  the algorithm of Section~\ref{section:algorithm}
to (\ref{experiments:chebyshev:3}) 
and compared the resulting phase function to (\ref{experiments:chebyshev:2}).
Figure~\ref{numerics:chebyshev:figure1} displays a plot of the
relative difference 
\begin{equation}
\frac{\left\|\alpha - \alpha_0\right\|_\infty}{\left\|\alpha_0\right\|_\infty}
\end{equation}
 between the exact phase function $\alpha_0$ and the phase function $\alpha$
obtained via the algorithm of Section~\ref{section:algorithm} as a function
of $\lambda$.
We observe that as $\lambda$ increases,
the difference between the phase function obtained via the algorithm
and the function $\lambda \arccos(t)$ 
 decays at an exponential
rate.

\label{section:experiment:chebyshev}
\end{subsection}

%
%
\begin{subsection}{Evaluation of Bessel functions.}

We compared the cost of evaluating Bessel functions
of integer order via the standard recurrence relation with that
of doing so using a nonoscillatory phase function.

We  denote by $J_\nu$ the Bessel function of the first kind
of order $\nu$.  It is a solution of the second order
differential equation
\begin{equation}
t^2 y''(t) + t y'(t)  + (t^2-\nu^2) y(t) = 0,
\label{experiments:bessel:1}
\end{equation}
which is brought into the standard form
\begin{equation}
\psi''(t) + \left(1-\frac{\lambda^2-1/4}{t^2} \right)\psi(t) = 0
\label{experiments:bessel:2}
\end{equation}
via the transformation
\begin{equation}
\psi(t) = \sqrt{t}\ y(t).
\label{experiments:bessel:3}
\end{equation}
An inspection of (\ref{experiments:bessel:2})  reveals that
$J_\nu$ is nonoscillatory on the interval
\begin{equation}
\left(0,\frac{1}{2}\sqrt{4\nu^2-1}\right)
\end{equation}
and oscillatory on the interval
\begin{equation}
\left(\frac{1}{2}\sqrt{4\nu^2-1},\infty\right).
\label{experiments:bessel:4}
\end{equation}
In addition to being a solution of a second order differential equation,
the Bessel function of the first kind
of order $\nu$ satisfies the three-term recurrence relation
\begin{equation}
J_{\nu+1}(t) = \frac{2\nu}{t} J_\nu(t) - J_{\nu-1}(t).
\label{experiments:bessel:5}
\end{equation}
The recurrence (\ref{experiments:bessel:5}) is numerically unstable in the forward
direction; however, when evaluated in the direction of decreasing
index, it yields a stable mechanism for evaluating Bessel functions
of integer order (see, for instance, Chapter 3 of \cite{NISTHandbook}).
These and many other properties of Bessel functions
are discussed in \cite{Watson}.

For each of $8$ values of $n$, we obtained an approximation
of the Bessel function  $J_n$  via the algorithm
of Section~\ref{section:algorithm} and compared its values
to those obtained through the recurrence relation 
at a collection of $1000$ randomly chosen points 
in the interval $[\frac{1}{2}\sqrt{4n^2-1},10n]$.  
The results of this experiment are shown in Table~\ref{experiments:bessel:table1}.
The phase function produced by the algorithm
of Section~\ref{section:algorithm} when  $n=10^4$
is shown in  Figure~\ref{experiments:bessel:figure1}.

\label{section:experiments:bessel}
\end{subsection}

\begin{subsection}{Evaluation of Legendre functions}

We used the algorithm of this paper to evaluate Legendre functions
of the first kind of various orders on the interval $[-1,1]$.

For each real number $\nu$,  we denote by $P_\nu$ the Legendre function
of the first kind of order $\nu$.    It is the solution of Legendre's equation
\begin{equation}
(1-t^2) y''(t) - 2t y'(t) + \nu(\nu+1) y(t) = 0
\ \ \ \mbox{for all} \ \ -1 \leq t \leq 1
\label{numerics:legendre:4}
\end{equation}
which is regular at the origin.  Letting
\begin{equation}
\psi(t) = \sqrt{1-t^2}\ y(t)
\label{numerics:legendre:6}
\end{equation}
in  Equation (\ref{numerics:legendre:4}) yields
\begin{equation}
\psi''(t) + 
\left(
\frac{1}{(1-t^2)^2} + \frac{\nu}{1-t^2} + \frac{\nu^2}{(1-t^2)^2}
\right)
\psi(t) 
= 0
\ \ \ \mbox{for all} \ \ -1 \leq t \leq 1,
\label{numerics:legendre:5}
\end{equation}
which is in a suitable form for the algorithm of Section~\ref{section:algorithm}.

We observe that the coefficient in equation (\ref{numerics:legendre:5}) is singular
at $\pm 1$,   which means that
phase functions for Legendre's equation are singular at $\pm 1$  as well. 
Accordingly,  in this experiment we used as input to the algorithm 
of Section~\ref{section:algorithm} a partition of the form
\begin{equation}
-\xi_{-k} < -\xi_{-k+1} < \ldots < \xi_{-1} < \xi_0 < \xi_1 < \ldots \xi_{k-1} <  \xi_{k},
\label{numerics:legendre:7}
\end{equation}
where $k=50$ and  $\xi_j$ is defined by the formula
\begin{equation}
\xi_j = 1 -  2^{-|j|}.
\label{numerics:legendre:8}
\end{equation}
The set $\{\xi_j\}$ is a ``graded mesh'' whose points
cluster near the singularities $\pm 1$ of the coefficient
in (\ref{numerics:legendre:5}).  Note that (\ref{numerics:legendre:7})
is not a partition of the entire interval $[-1,1]$  but rather a partition
of $[-b,b]$, where
\begin{equation}
b = 1-2^{-50}.
\end{equation}
%
Functions were represented using $16^{th}$ order Chebyshev expansions on each of the
intervals $[\xi_{j},\xi_{j+1}]$.

For each of $11$ values of $\nu$, the algorithm of this paper was applied
to Equation (\ref{numerics:legendre:5}) and the solution evaluated
at a collection of $1000$ randomly chosen points on the interval $[-1,1]$.
In order to assess the 
the error in each obtained solution, we constructed
a reference solutions by performing the calculations a second time using
extended precision (Fortran REAL*16) arithmetic.
The results are reported in Table~\ref{numerics:legendre:table1}.  Each row
corresponds to one value of $\nu$ and reports the time required to construct
the nonoscillatory phase functions, the average time required to 
evaluate the Legendre function of the first kind of order $\nu$ 
using this nonoscillatory phase function, and the maximum observed absolute error.
Figure~\ref{numerics:legendre:figure1} depicts the solution of the logarithm
form of Kummer's equation obtained by the algorithm of this paper when $\nu = \pi \cdot 10^5$.

\label{section:experiments:legendre}
\end{subsection}

\begin{subsection}{Evaluation of prolate spheriodal wave functions}

We used the  algorithm of Section~\ref{section:algorithm}
to evaluate prolate spheriodal wave functions of order 0 and we compared its
performance   with  that of the Osipov-Rokhlin algorithm
\cite{Osipov-Rokhlin}.

Suppose that $c > 0$ is a real number.  Then there exists a sequence 
\begin{equation}
0 < \chi_{c,0} < \chi_{c,1} < \chi_{c,2} < \cdots
\label{experiments:prolates:1}
\end{equation}
of positive real numbers such for each nonnegative integer $n$,
the second order differential equation
\begin{equation}
(1-t^2) \psi''(t) - 2 t \psi'(t) + (\chi_{c,n} - c^2 t^2 )\psi(t) = 0
\label{experiments:prolates:2}
\end{equation}
has a continuous solution on the interval $[-1,1]$.
These solutions are the prolate spheriodial wave functions of 
order $0$ associated with the parameter $c$.  We denote them
by 
\begin{equation}
\psi_{c,0}(t), \psi_{c,1}(t), \psi_{c,2}(t), \ldots.
\label{experiments:prolates:3}
\end{equation}
The monograph \cite{Osipov-Rokhlin-Xiao}
contains a detailed discussion of the prolate spheriodal wave functions
of order $0$.

By introducing  the function
\begin{equation}
\varphi(t) = \psi(t) \sqrt{1-t^2}
\label{experiments:prolates:4}
\end{equation}
into (\ref{experiments:prolates:2}), we bring it into the form
\begin{equation}
\varphi''(t) + 
\left(
\frac{1}{(1-t^2)^2} + \frac{\chi_{c,n}}{1-t^2} - c^2t^2
\right) 
\varphi(t) = 0.
\label{experiments:prolates:5}
\end{equation}
An inspection of (\ref{experiments:prolates:5})  reveals that
 the coefficient in (\ref{experiments:prolates:5})
is nonnegative on the interval $[-1,1]$
when $\chi_n \geq c^2$.


For several values of $c$ and $\chi_{n,c} > c^2$, we evaluated the prolate
spheriodial wave function $\psi_{c,n}$ at a collection of
$100$  randomly chosen points 
in the interval $[-1,1]$
by applying the algorithm of Section~\ref{section:algorithm} to (\ref{experiments:prolates:5})
and via the  Osipov-Rokhlin algorithm.      
Table~\ref{numerics:prolate:table} presents the results
and Figure~\ref{numerics:prolate:figure1} shows a 
plot of $\alpha(t) - ct$, where 
$c= 10^5$, $n = 63769$, $\chi_{c,n} = 1.00060408908491\e{10}$
and $\alpha$ is the nonoscillatory
phase function for Equation~(\ref{experiments:prolates:5})
produced by the algorithm of Section~\ref{section:algorithm}.



\label{section:experiments:prolates}
\end{subsection}

\label{section:experiments}
\end{section}

\begin{section}{Acknowledgments}
The author would like to thank Vladimir Rokhlin for reading a draft
of this manuscript and for his many helpful suggestions,
and Andrei Osipov for providing
his code for evaluating prolate spheriodal wave functions.
James Bremer was supported by a fellowship from the Alfred P. Sloan
Foundation, and by  National Science Foundation grant DMS-1418723.
\end{section}

\begin{section}{References}
\bibliographystyle{acm}
\bibliography{numerics}
\end{section}

\vfill\eject
\begin{figure}[h!!]
\begin{center}
\includegraphics[width=.45\textwidth]{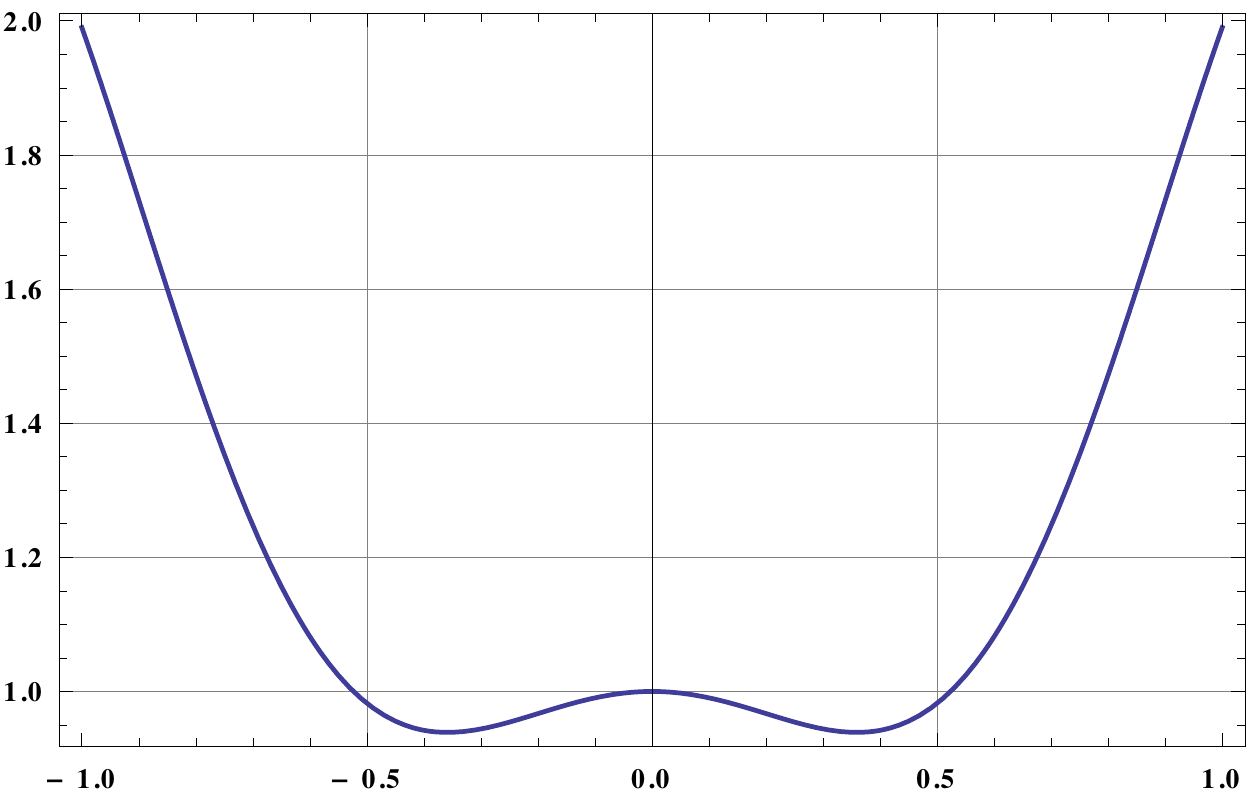}
\hfil
\includegraphics[width=.45\textwidth]{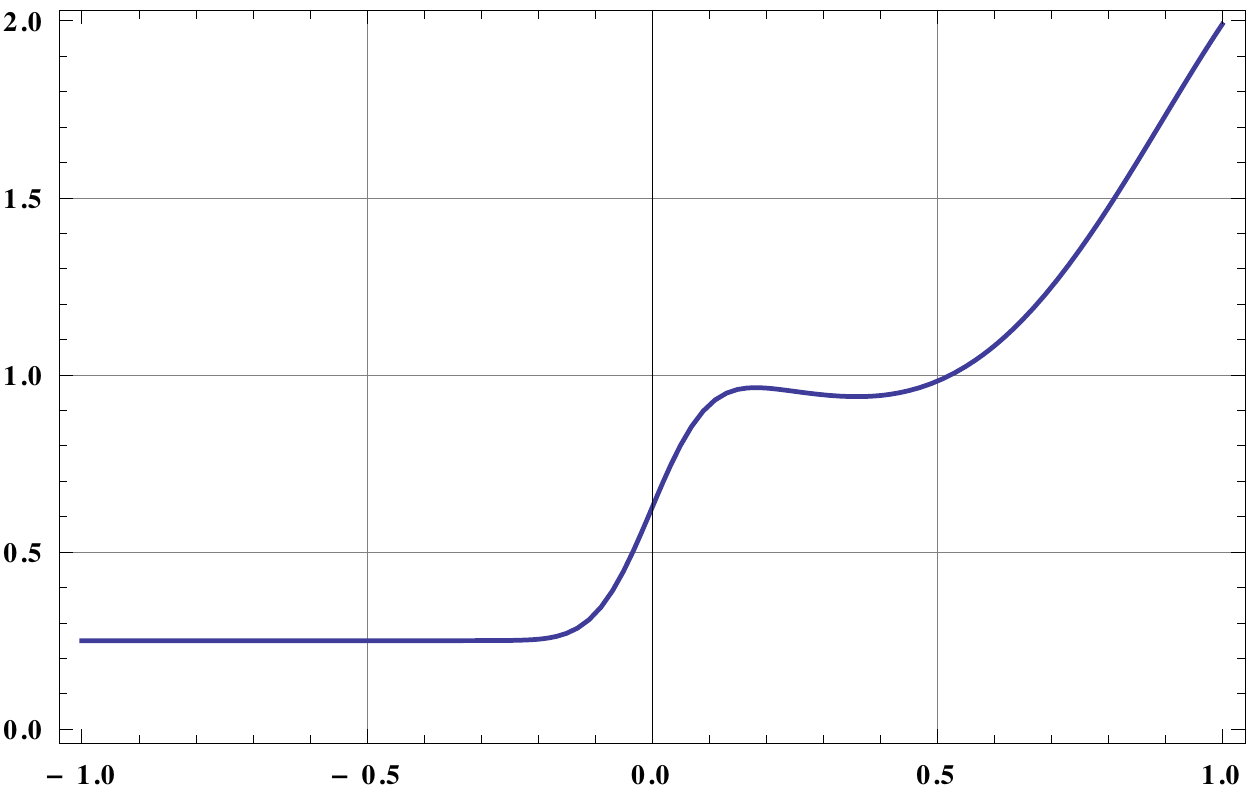}
\caption
{
The function $q$ defined by formula (\ref{experiments:simple:q}) 
in Section~\ref{section:experiments:simple} (left) and the
windowed version $\tilde{q}$ of $q$ (right).
}
\label{figure:simple:q}
\end{center}
\end{figure}

\vfil

\begin{figure}[h!!]
\begin{center}
\includegraphics[width=.45\textwidth]{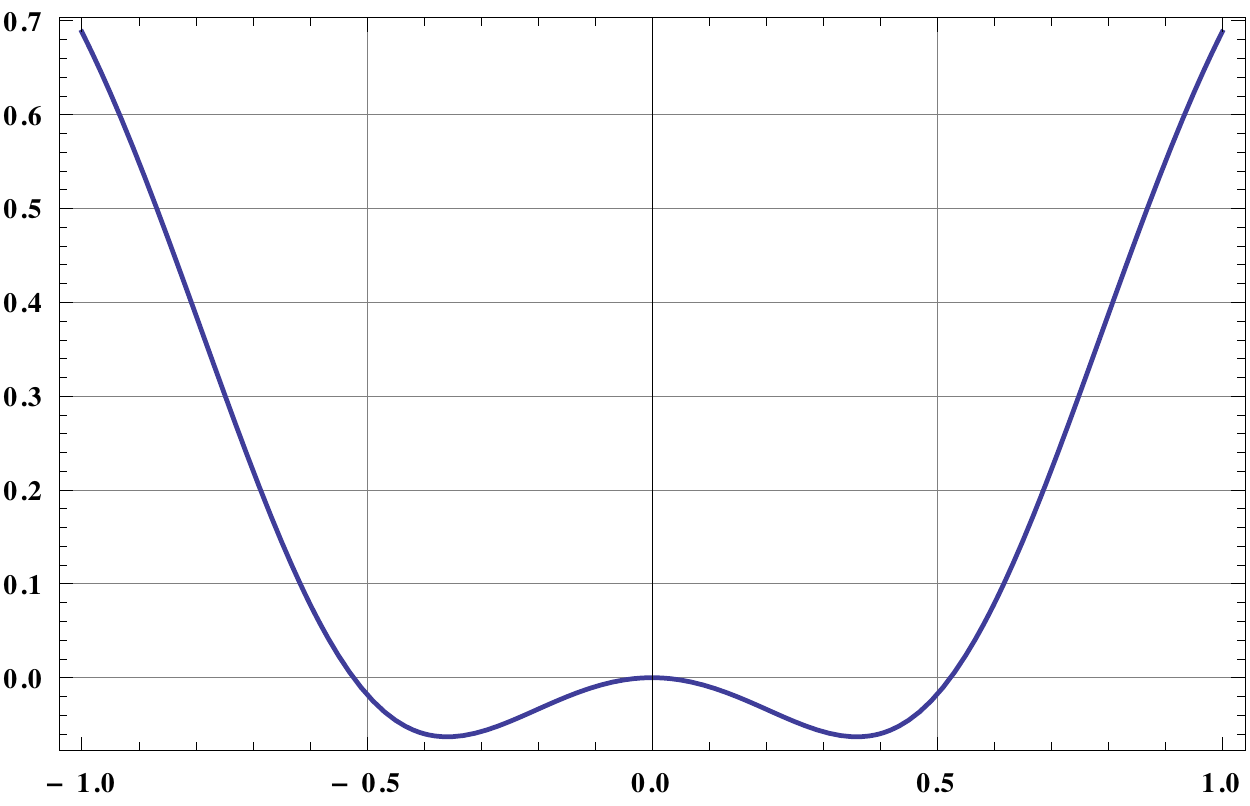}
\hfil
\includegraphics[width=.45\textwidth]{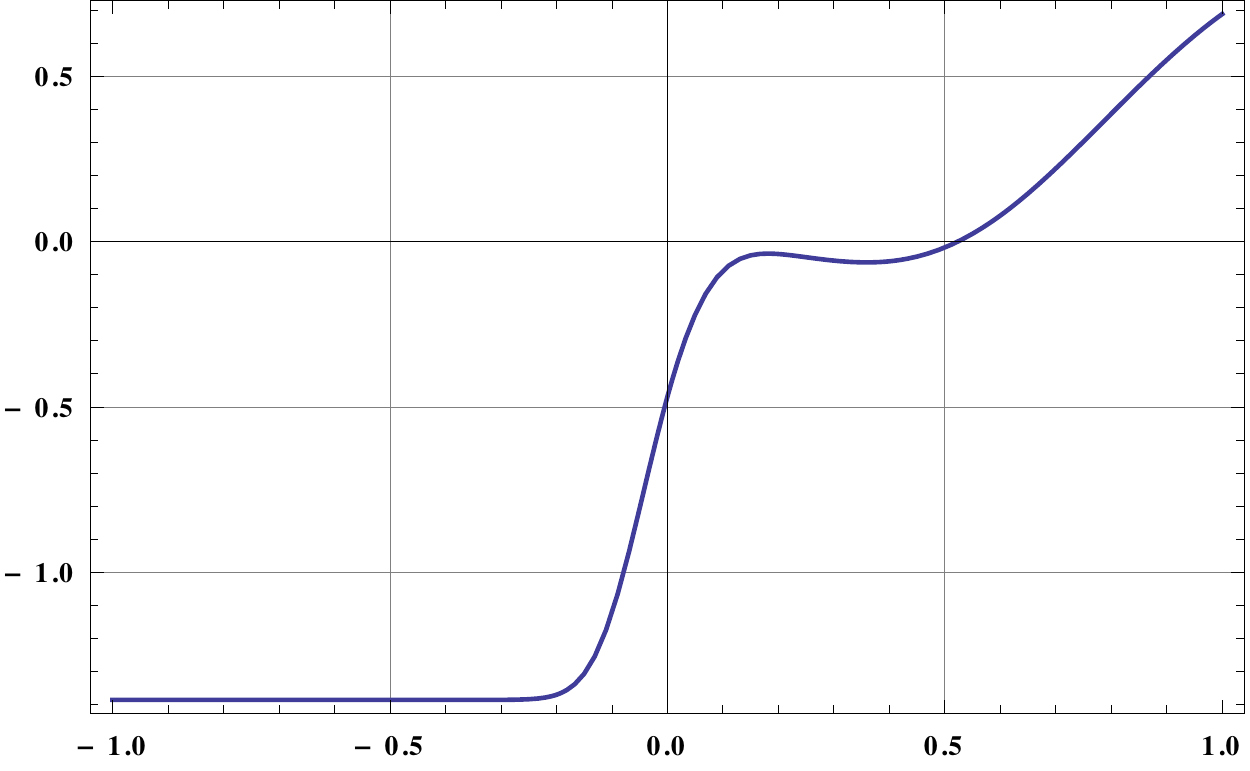}
\caption
{
Plots of the solution $r$ of the logarithm form of Kummer's
equation associated with the ordinary differential equation 
(\ref{experiments:simple:ode})  
in Section~\ref{section:experiments:simple} when $\lambda = 10^7$ (left) 
and the function $r_1$ constructed as an intermediate
step by the algorithm of Section~\ref{section:algorithm} (right).
}
\label{figure:simple:r}
\end{center}
\end{figure}
\vfill

\begin{table}[h!!]
\begin{center}
\begin{tabular}{cccc}
\multirow{2}{*}{$\lambda$}  & Phase function  & Avg. phase function & Maximum \\
                            & construction time & evaluation time & error \\
\midrule
\addlinespace[.5em]
$10^1$ & $7.25\e{-02}$ & $1.55\e{-06}$ & $6.93\e{-14}$\\
\addlinespace[.25em]
$10^2$ & $9.17\e{-02}$ & $1.58\e{-06}$ & $5.39\e{-13}$\\
\addlinespace[.25em]
$10^3$ & $6.74\e{-02}$ & $1.55\e{-06}$ & $3.01\e{-12}$\\
\addlinespace[.25em]
$10^4$ & $6.73\e{-02}$ & $1.55\e{-06}$ & $4.82\e{-11}$\\
\addlinespace[.25em]
$10^5$ & $6.72\e{-02}$ & $1.59\e{-06}$ & $3.23\e{-10}$\\
\addlinespace[.25em]
$10^6$ & $6.66\e{-02}$ & $1.64\e{-06}$ & $5.15\e{-09}$\\
\addlinespace[.25em]
$10^7$ & $8.60\e{-02}$ & $1.61\e{-06}$ & $3.64\e{-08}$\\
\end{tabular}
\caption{
The accuracy and running time of the algorithm of this paper
when applied to to the initial value problem
(\ref{experiments:simple:ode}) of Section~\ref{section:experiments:simple}.
}
\label{numerics:simple:table}
\end{center}
\end{table}

\vfil

\begin{figure}[h!!]
\begin{center}
\includegraphics[width=.9\textwidth]{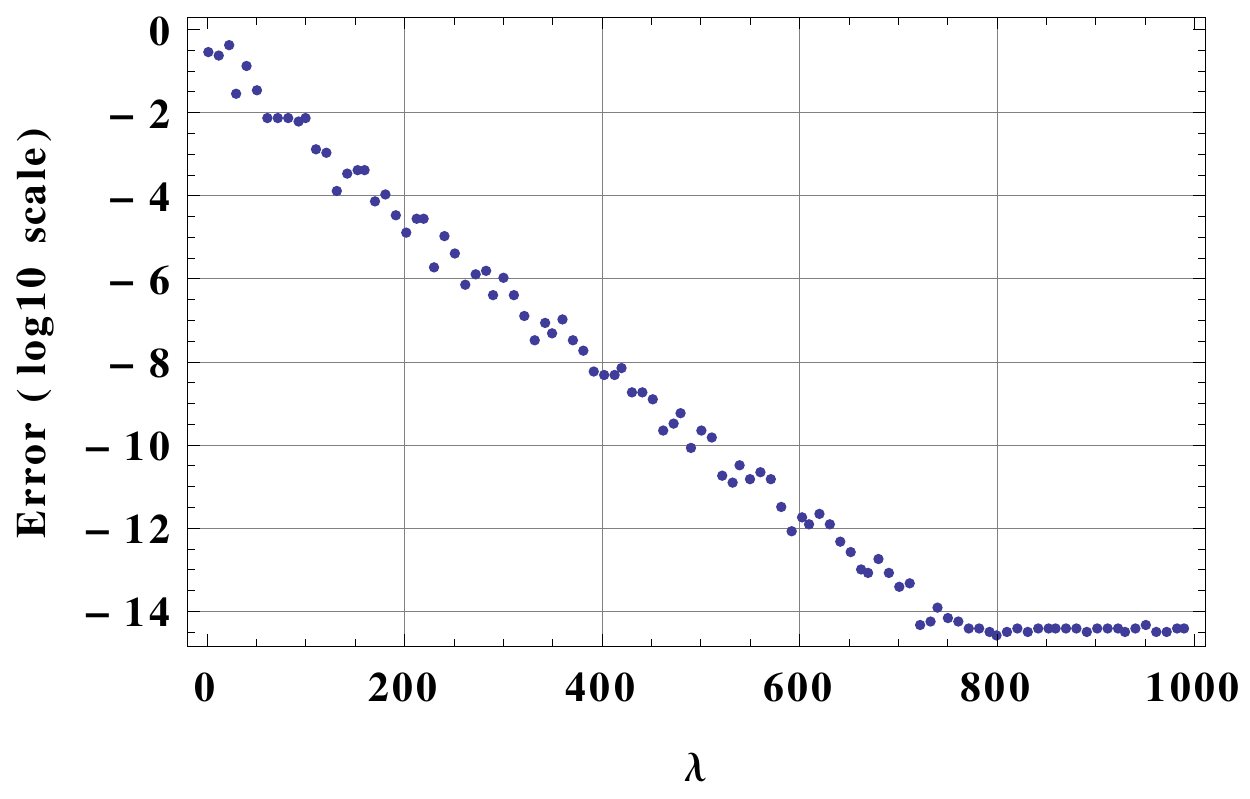}
\caption{
A plot of the base-$10$ logarithm of the relative difference between
phase function obtained by applying the algorithm
of this paper to Chebyshev's equation (\ref{experiments:chebyshev:3})
and the  well-known nonoscillatory phase function 
$ \lambda\arccos(t)$ for Chebyshev's equation.
}
\label{numerics:chebyshev:figure1}
\end{center}
\end{figure}


\vfil
\begin{table}[h!!!!!!!]
\begin{center}
\begin{tabular}{ccccc}
\multirow{2}{*}{$n$}
&   Phase function  & Avg. phase function & Avg. recurrence  & Maximum \\
& construction time & evaluation time     & evaluation time  & error\\
\midrule
\addlinespace[.5em]
$10^1$    & 1.70\e{-02} secs    & 2.24\e{-07} secs & 1.40\e{-06} secs & 1.58\e{-14} \\
\addlinespace[.25em]
$10^2$    & 2.27\e{-02} secs    & 2.06\e{-07} secs & 6.17\e{-06} secs & 1.75\e{-14} \\
\addlinespace[.25em]
$10^3$    & 1.62\e{-02} secs    & 2.23\e{-07} secs & 4.60\e{-05} secs & 4.62\e{-14} \\
\addlinespace[.25em]
$10^4$    & 1.65\e{-02} secs    & 2.24\e{-07} secs & 4.29\e{-04} secs & 3.52\e{-13} \\
\addlinespace[.25em]
$10^5$    & 1.62\e{-02} secs    & 2.29\e{-07} secs & 4.12\e{-03} secs & 4.70\e{-13} \\
\addlinespace[.25em]
$10^6$    & 1.66\e{-02} secs    & 2.65\e{-07} secs & 4.20\e{-02} secs & 1.66\e{-12} \\
\addlinespace[.25em]
$10^7$    & 2.94\e{-02} secs    & 2.69\e{-07} secs & 4.22\e{-01} secs & 3.88\e{-11} \\
\addlinespace[.25em]
$10^8$    & 6.42\e{-01} secs    & 6.39\e{-07} secs & 4.33\e{+00} secs & 3.91\e{-11} \\
\end{tabular}
\caption{
A comparison of the time required to evaluate the Bessel function $J_n$
using the standard recurrence relation 
with that required to evaluate it  using a nonoscillatory phase function.
The recurrence relation approach scales as $O(n)$ in the order $n$ 
while the time required by the phase function method is $O(1)$.
}
\label{experiments:bessel:table1}
\end{center}
\end{table}

\vfill

\begin{figure}[h!!]
\begin{center}
\includegraphics[width=.9\textwidth]{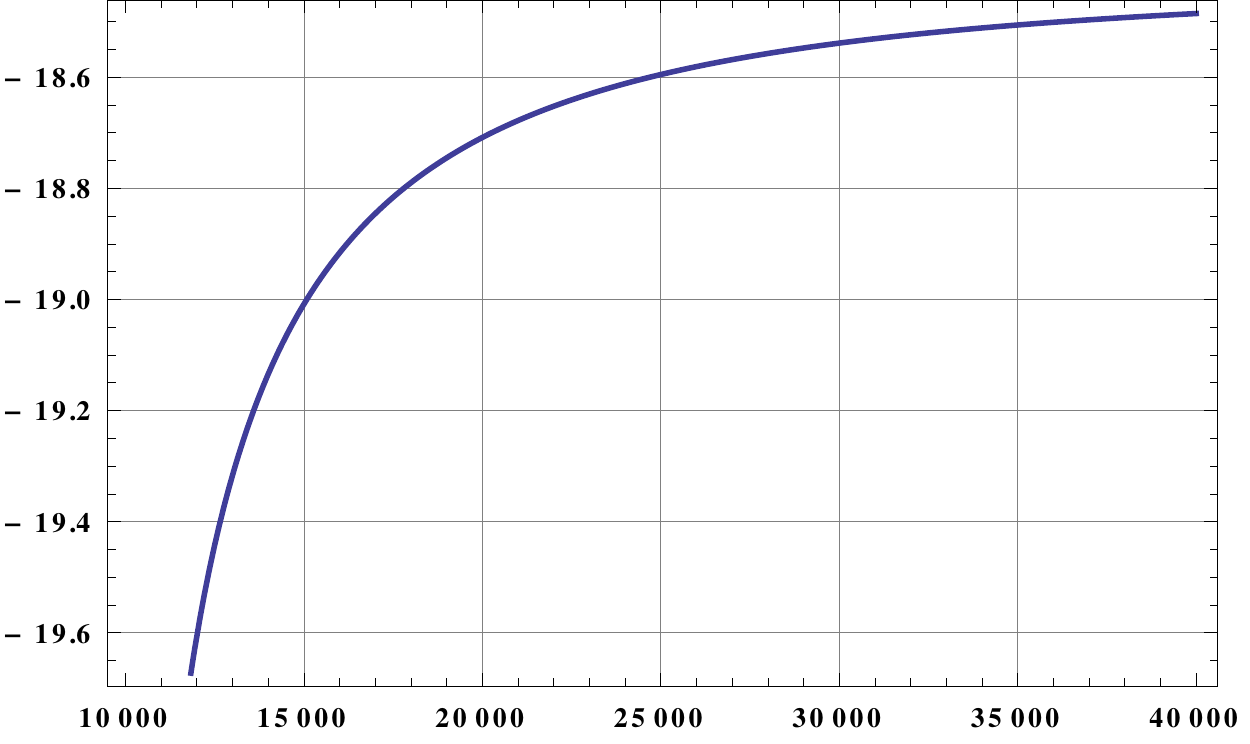}
\caption{
A plot of the nonoscillatory phase function for 
Bessel's equation (\ref{experiments:bessel:1}) when $n = 10^4$.
}
\label{experiments:bessel:figure1}
\end{center}
\end{figure}

\vfil

\begin{table}[h!!]
\begin{center}
\begin{tabular}{cccc}
\multirow{2}{*}{$\nu$}    &   Phase function & Avg. phase function & Maximum \\
                          & construction time & evaluation time    &  error\\
\midrule
\addlinespace[.5em]
$\pi \cdot 10^4$      & 2.92\e{-03} secs     & 3.13\e{-07} secs & 1.32\e{-13} \\
\addlinespace[.25em]
$\pi \cdot 10^5$      & 2.60\e{-03} secs     & 3.84\e{-07} secs & 3.24\e{-13} \\
\addlinespace[.25em]
$\pi \cdot 10^6$      & 2.46\e{-03} secs     & 4.63\e{-07} secs & 1.09\e{-12} \\
\addlinespace[.25em]
$\pi \cdot 10^7$      & 2.97\e{-03} secs     & 3.67\e{-07} secs & 3.21\e{-12} \\
\addlinespace[.25em]
$\pi \cdot 10^8$      & 2.53\e{-03} secs     & 3.61\e{-07} secs & 1.35\e{-11} \\
\addlinespace[1.0em]
$\sqrt{2} \cdot 10^4$      & 5.61\e{-03} secs     & 3.20\e{-07} secs & 7.22\e{-13} \\
\addlinespace[.25em]
$\sqrt{2} \cdot 10^5$      & 5.43\e{-03} secs     & 3.54\e{-07} secs & 3.32\e{-12} \\
\addlinespace[.25em]
$\sqrt{2} \cdot 10^6$      & 2.49\e{-03} secs     & 4.26\e{-07} secs & 1.13\e{-12} \\
\addlinespace[.25em]
$\sqrt{2} \cdot 10^7$      & 2.61\e{-03} secs     & 3.84\e{-07} secs & 2.79\e{-12} \\
\addlinespace[.25em]
$\sqrt{2} \cdot 10^8$      & 5.43\e{-03} secs     & 3.67\e{-07} secs & 8.65\e{-12} \\
\addlinespace[.25em]
$\sqrt{2} \cdot 10^9$      & 4.94\e{-03} secs     & 3.95\e{-07} secs & 2.85\e{-11} \\
\end{tabular}
\caption{
The results obtained by applying the 
algorithm of Section~\ref{section:algorithm} to Legendre's 
differential equation (\ref{numerics:legendre:4}).  We observe that
the running time is independent of $\nu$, but that some accuracy
is lost when evaluating Legendre functions of large orders.
}
\label{numerics:legendre:table1}
\end{center}
\end{table}

\begin{figure}[h!!]
\begin{center}
\includegraphics[width=.9\textwidth]{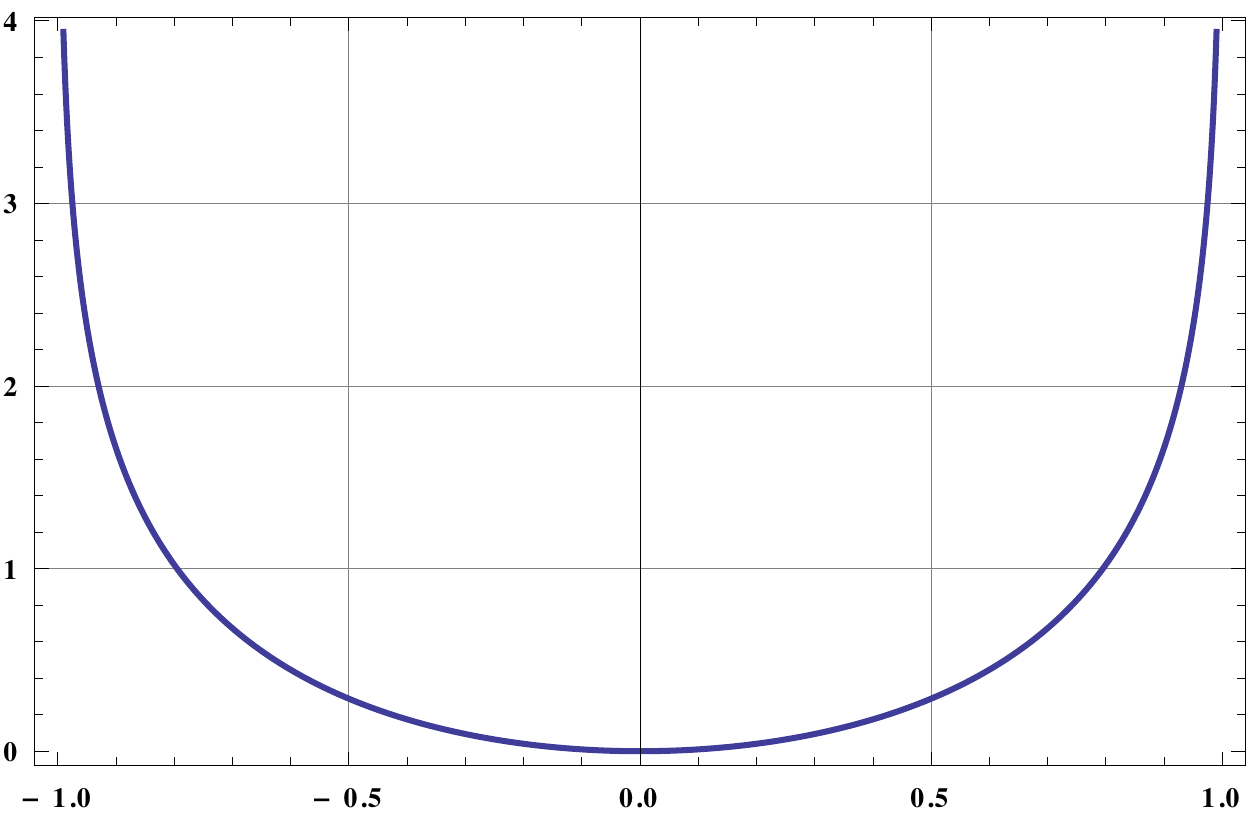}
\caption{
A plot of the nonoscillatory solution of the logarithm form of Kummer's equation
associated with 
Legendre's differential equation (\ref{numerics:legendre:4}) when $\nu = \pi 10^5$.
This function has singularities  at the points $\pm 1$ and is represented using a graded
mesh which becomes dense near them.
}
\label{numerics:legendre:figure1}
\end{center}
\end{figure}

\vfil

\vfill
\begin{figure}[h!!]
\begin{center}
\includegraphics[width=.9\textwidth]{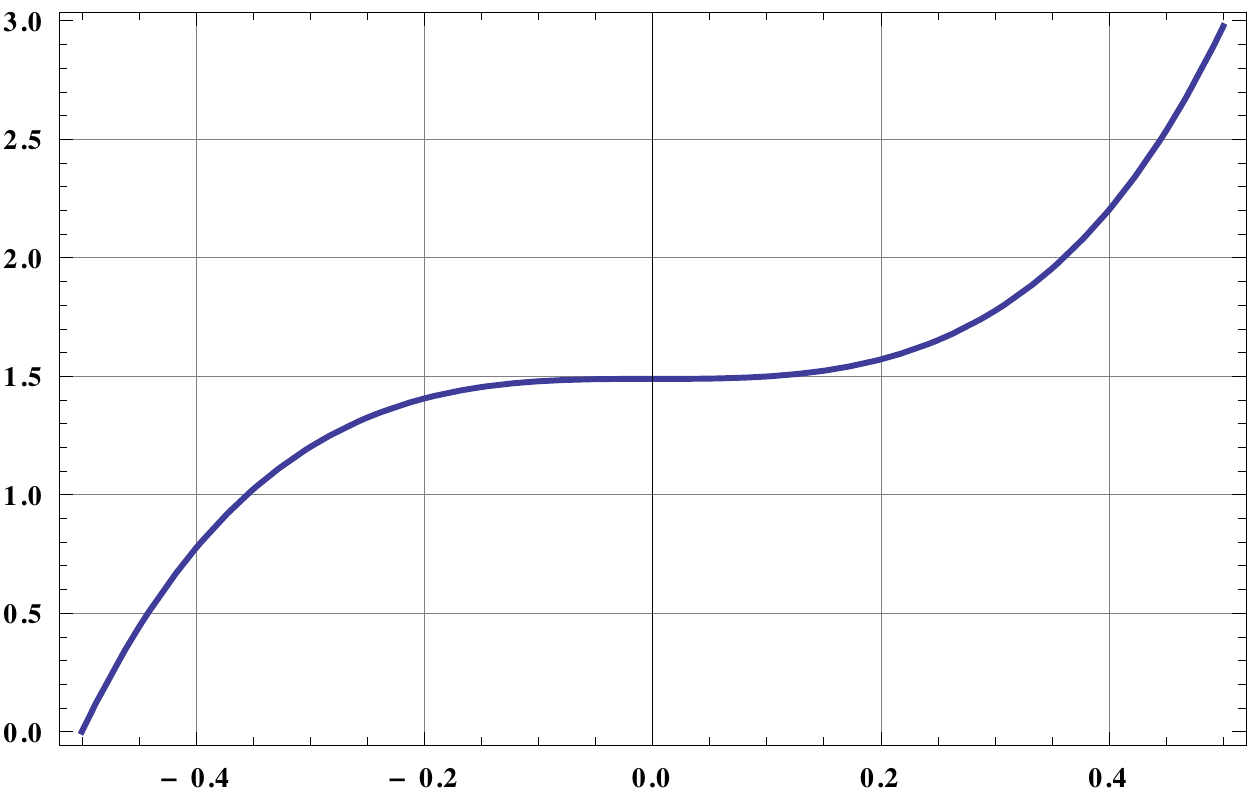}
\caption{
A plot of the function $\alpha(t) - c t$, where $\alpha$ 
is nonoscillatory phase function associated with equation
(\ref{experiments:prolates:5}), $c=10^5$, $n = 63769$ ,  and $\chi_{n,c} =1.00060408908491\e{10}$.
}
\label{numerics:prolate:figure1}
\end{center}
\end{figure}

\begin{sidewaystable}[h!]
\begin{center}
\begin{tabular}{rrccccc}
\multirow{3}{*}{$c$}& \multirow{3}{*}{$n$} &  \multirow{3}{*}{$\chi_{c,n}$} &                  & Avg. phase         & Average                         & \\
                    &  &                                  &         Phase function                         & function            & Osipov-Rokhlin    & Maximum \\
                    &  &                          &             construction time                & evaluation time     & evaluation time  & error  \\
\midrule
\addlinespace[.5em]
$10^4$ & $12904$ & $2.18416195669669 \e{+08}$ & 3.30\e{-02} secs     & 3.22\e{-07} secs & 1.65\e{-04} secs & 3.24\e{-11} \\
\addlinespace[.25em]
$10^5$ & $63769$    & $1.00060408908491 \e{+10}$ & 3.29\e{-02} secs     & 3.29\e{-07} secs & 1.11\e{-03} secs &1.67\e{-10} \\
\addlinespace[.25em]
$10^5$ & $95653$    & $1.44996988449419 \e{+10}$ & 3.32\e{-02} secs     & 3.19\e{-07} secs & 1.34\e{-03} secs &2.00\e{-10} \\
\addlinespace[.25em]
$10^6$ & $636748$   &  $1.00005993679849 \e{+12}$ & 3.31\e{-02} secs     & 3.70\e{-07} secs & 1.14\e{-02} secs & 1.77\e{-09} \\
\addlinespace[.25em]
$10^6$ & $1910244$   & $4.15761057502686 \e{+12}$ & 3.30\e{-02} secs     & 3.60\e{-07} secs & 2.34\e{-02} secs & 5.19\e{-09} \\
\addlinespace[.25em]
$10^7$ & $12732696$ &  $2.14063766093698 \e{+14}$ & 3.08\e{-02} secs     & 3.60\e{-07} secs & 1.59\e{-01} secs & 3.74\e{-08} \\
\addlinespace[.25em]
$10^8$ & $127324798$ & $2.14057058422053 \e{+16}$ & 4.34\e{-02} secs     & 4.10\e{-07} secs & 1.59\e{+00} secs & 3.30\e{-07} \\
\addlinespace[.25em]
$10^8$ & $190986447$ & $4.15616226165926 \e{+16}$ & 4.32\e{-02} secs      & 3.60\e{-07} secs & 2.22\e{+00} secs & 4.35\e{-07} \\
\addlinespace[.25em]
$10^9$ & $636619965$ & $1.00000005945416 \e{+18}$ & 4.31\e{-02} secs     & 4.01\e{-07} secs & 1.15\e{+01} secs & 8.61\e{-07} \\
\addlinespace[.25em]
\end{tabular}
\end{center}
\caption{A comparison of the results obtained by using the algorithm of this paper to evaluate prolate spheriodial 
wave functions of order $0$ with those obtained via the Osipov-Rokhlin algorithm \cite{Osipov-Rokhlin}.}
\label{numerics:prolate:table}
\end{sidewaystable}

\end{document}